\documentclass[10pt]{amsart}

\usepackage{hyperref}
\usepackage{tasks}

\usepackage{blindtext}

\usepackage{setspace}
\usepackage{marginnote}
\usepackage{color}
\usepackage{amsrefs}
\usepackage{graphicx}  
\usepackage[top=3.5cm, bottom=2.5cm, outer=2.5cm, inner=2.5cm]{geometry}

\hypersetup{linkcolor=red,
citecolor=black
}
\usepackage{amsmath}

\newtheorem{thm}{Theorem}[section]
\newtheorem{lem}[thm]{Lemma}

\newtheorem{cor}[thm]{Corollary}

\theoremstyle{definition}

\theoremstyle{remark}
\newtheorem{re}[thm]{Remark}

\numberwithin{equation}{section}

\newcommand*{\affaddr}[1]{#1}

\usepackage[english]{babel}
\usepackage{paralist}
\usepackage{amsthm}
 
\theoremstyle{definition}
\newtheorem{definition}{Definition}[section]
 
\theoremstyle{remark}

 \usepackage{subfigure}
 \usepackage{bm,amsmath}

\newcommand{\tran}{^{\mathstrut\scriptscriptstyle\top}} 
 
\newcommand{\RN}[1]{%
  \textup{\uppercase\expandafter{\romannumeral#1}}%
}

\newcommand{\Rey}{\mathcal{R}e }
\newcommand{\x}{\cdot}
\newcommand{\grad}{\nabla}
\allowdisplaybreaks[4]
\begin{document}

\title{Analysis of mesh effects on turbulent flow statistics}


\author{%
Ali Pakzad\\
\affaddr{alp145@pitt.edu\\Department of Mathematics}\\
\affaddr{ University of Pittsburgh}%
}

\maketitle
\setcounter{tocdepth}{1}

\begin{abstract}
Turbulence models, such as the Smagorinsky model herein, are used to represent the energy lost from resolved to under-resolved scales due to the energy cascade (i.e. non-linearity). Analytic estimates of the energy dissipation rates of a few turbulence models have recently appeared, but none (yet) study the  energy dissipation restricted to resolved scales, i.e. after spatial  discretization with $h >$ micro scale. We do so herein for the Smagorinsky model.  Upper bounds  are derived on the \textit{computed} time-averaged energy dissipation rate, $\langle  \varepsilon (u^h)\rangle$, for an  under-resolved mesh $h$ for  turbulent shear flow. For  coarse  mesh size $ \mathcal{O}(\Rey^{-1}) < h < L $, it is proven, 
$$
\langle  \varepsilon (u^h)\rangle\leq \big[   (\frac{C_s\, \delta}{h})^2+ \frac{L^5}{(C_s \delta)^4\,h}+\frac{L^{\frac{5}{2}}}{(C_s\, \delta)^{4}}\, {h^{\frac{3}{2}}}\big]\, \frac{U^3}{L}, 
$$
where $U$ and $L$ are global velocity and length scale and $C_s$ and $\delta$ are  model parameters.
This upper bound  being independent of the viscosity at high Reynolds number, is  in accord with the the equilibrium dissipation law (Kolmogrov's conventional turbulence theory). This estimate  suggests over-dissipation   for any  of $C_s>0$ and $\delta>0$,  consistent with numerical evidence on the effects of model viscosity (without wall damping function). Moreover, the analysis indicates that the turbulent boundary layer is a  more important length scale for shear flow than the Kolmogorov microscale.
\end{abstract}

\section{introduction}

Because of the limitations of computers, we are forced to struggle with the meaning of  under-resolved flow simulations. Turbulence models are introduced to account for sub-mesh scale effects, when solving fluid flow problems numerically.  One key in getting a good approximation for a turbulent model is to correctly calibrate the energy dissipation $ \varepsilon (u)$ in the model on an under-resolved mesh \cite{L02}.  The energy dissipation rates of various turbulent models have been analyzed assuming infinite resolution (i.e. for the continuous model e.g. \cite{DLPRSZ18},  \cite{DC92}, \cite{L16}, \cite{L02}, \cite{M94}, \cite{A.P},  \cite{A.P19}, \cite{W00} and   \cite{W97}). The  question explored herein is: What is the  time-averaged energy dissipation rate $\langle  \varepsilon (u^h)\rangle$,  when $u^h$ is an approximation of $u$ on a given mesh $h$ (fixed computational cost) for the Smagorinsky model?

 Numerical simulations of turbulent flows  include only the motion of eddies above some length scale $\delta$ (which depends on attainable resolution).  The state of motion in turbulence is too complex to allow for a detailed description of the fluid velocity.  Benchmarks are required to validate simulations' accuracy. Two natural benchmarks for numerical simulations are the kinetic energy in the eddies of size $> \mathcal{O}(\delta)$ and the energy dissipation rates of eddies in size $> \mathcal{O}(\delta)$ (e.g. \cite{MKM98} and \cite{IF04}).  Motivated from here,  we focus on the calculation of the  time-averaged energy dissipation rates of motions larger than $\delta$. In the large-eddy simulation (LES), predicting how flow statistics (e.g. time-averaged) depend on the grid size $h$ and the filter size $\delta$ is listed as one of  ten important questions by Pope \cite{P04}.

Consider the Smagorinsky model (\ref{SM}) subject to a boundary-induced shear (\ref{BC}) in the flow domain $\Omega = (0,L)^3$,
\begin{equation}\label{SM}
\begin{split}
  u_t+u &\cdot \nabla u -\nu \Delta u + \nabla p - \grad \x ((C_s \delta)^2 |\grad u| \grad u)=0,\\
 &  \grad \cdot u =0 \hspace{10pt} \hspace{10pt}\mbox{and}\hspace{10pt} u(x,0)=u_0  \hspace{10pt}   \mbox{in}\,\, \Omega.
 \end{split}
   \end{equation}
 In (\ref{SM})  $u$ is the velocity field, $p$ is the pressure, $\nu$ is  the kinematic viscosity, $\delta$ is the  turbulence-resolution length scale (associated with the mesh size $h$) and  $C_s \simeq 0.1$  is the standard model parameter (see Lilly \cite{L67} for more details).  We consider shear flow, containing a turbulent boundary layer, in a domain that is simplified to permit more precise analysis. $L$-periodic boundary conditions in $x$ and $y$ directions are imposed. $z=0$ is a fixed wall and the wall $z=L$ moves with velocity $U$, 

 \begin{equation} \label{BC}
\begin{split}
 L-\mbox{periodic boundary con}&\mbox{dition in  }\, x \mbox{ and }\, y \mbox{ direction,}\\
   u(x,y,0,t)=(0,0,0)\tran \hspace{18pt}&\mbox{and}\hspace{18pt} u(x,y,L,t)=(U,0,0)\tran.
\end{split}
\end{equation}

\begin{figure}[b]
\includegraphics[scale=0.5]{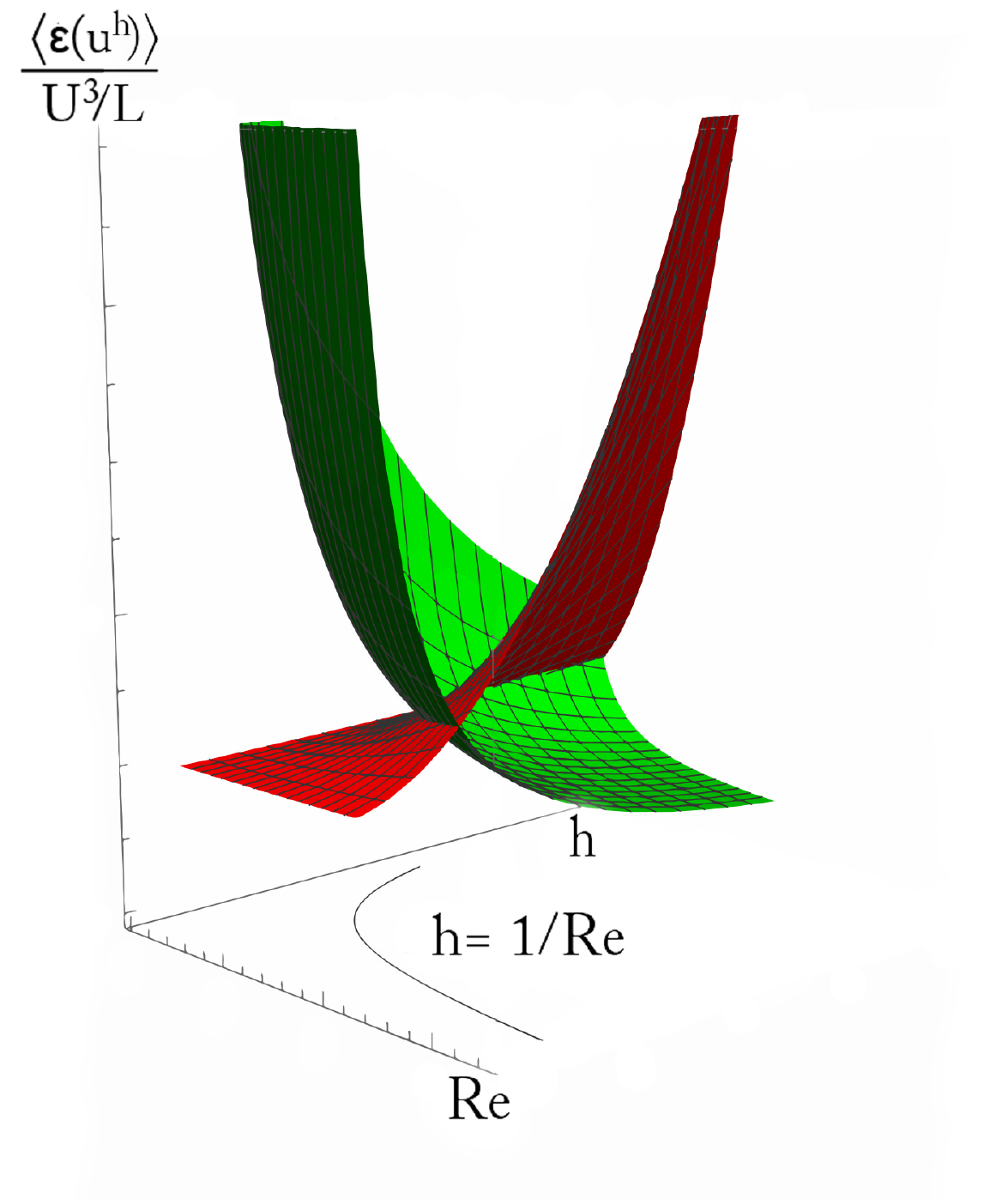}
\caption{Dissipation vs.  mash size $h$ and  the Reynolds number $\Rey$.}
\label{3plots}
\end{figure} 

 To  disregard the effects of the time discretization on dissipation, we consider (\ref{SM})  continuous in time and  discretized in space  by a standard finite element method (see (\ref{FEM-SM}))  since the effects of the spatial  discretization, and not the time discretization,  on dissipation is the subject of study here. In Theorem \ref{thm1}, we first estimate  $\langle  \varepsilon (u^h)\rangle$  for fine enough mesh that no model to be necessary ($ h < \mathcal{O}(\Rey^{-1})\, L$).  Since models are used for meshes coarser than a DNS, we next investigate  (\ref{SM}) on an arbitrary coarse mesh. Theorem \ref{thm2} presents the analysis of $\langle  \varepsilon (u^h)\rangle$ on an under-resolved mesh.  The results in Theorems \ref{thm1} and \ref{thm2} can be summarized as,

 \vspace{12pt}

\begin{equation} \label{summery}
\frac{\langle  \varepsilon (u^h)\rangle}{U^3/L}  \simeq
\left\{
	\begin{array}{ll}
       1+  (\frac{c_s \delta}{L})^2\, \Rey ^2   & \mbox{ for }\hspace{10pt} h < \mathcal{O}(\Rey^{-1}) \, L   \vspace{7pt}\\
     
       		  \frac{1}{\Rey} \,\frac{L}{h}+   (\frac{C_s\, \delta}{h})^2+ \frac{L^5}{(C_s \delta)^4\,h}+\frac{L^{\frac{5}{2}}}{(C_s\, \delta)^{4}}\, {h^{\frac{3}{2}}} & \mbox{ for  } \hspace{10pt} h \geq \mathcal{O}(\Rey^{-1}) \, L
	\end{array}.
\right.
\end{equation}

 \vspace{12pt}

Considering $\delta = h$ in (\ref{summery}), which is the common choice in  LES \cite{P04},  we speculate that  $\frac{\langle  \varepsilon (u^h)\rangle}{U^3/L} $ varies with $\Rey$ and $h$ as depicted qualitatively in Figure \ref{3plots}.

 In the \textit{under-resolved} case $(h \geq \Rey^{-1} \, L)$, the upper bound  (\ref{summery}) being independent of the viscosity at high Reynolds number, is in accord with the Kolmogrov's conventional turbulence theory. He  argued that at large Reynolds number, the energy dissipation rate per unit volume should be independent of the kinematic viscosity and hence, by a dimensional consideration, the energy dissipation rate per unit volume must take the form constant times $\frac{U^3}{L}$. On the other hand, the weak $\Rey$ dependence in the \textit{under-resolved} case (that vanishes as $\Rey\rightarrow \infty$) is consistent with  Figure \ref{Tritton}, adopted from \cite{F95}. This figure shows how dissipation coefficient varies with Reynolds number for a  circular cylinder based on experimental data quoted in \cite{T88}.  Tritton \cite{T88} observed  that dissipation coefficient  behaves as $\Rey^{-1}$  for  small $\Rey$,  and at high $\Rey$  stays approximately constant. Theorem \ref{thm2}  is also consistent  with the recent results in \cite{MBYL15} derived through structure function theories of turbulence.

 For the fully-resolved case $(h < \Rey^{-1} \, L)$, the estimate (\ref{summery}) is consistent as $\Rey \rightarrow \infty$ and $\delta = h \simeq \mathcal{O}(\Rey^{-1}) \rightarrow 0$  with the Kolmogrov's conventional turbulence theory. It is also in accord with the rate proven for the Navier-Stokes equations \cite{DC92} and the Smagorinsky model \cite{L02}.

 Corollaries \ref{Cor1} and \ref{Cor2} show that the estimate (\ref{summery}) \textit{suggests} over-dissipation of the model for any choice of $C_s>0$ and $\delta >0$. In  other words, the constant $C_s$  causes excessive damping of large-scale fluctuation, which agrees with the computational experience  (p.247 of Sagaut \cite{Sag} and \cite{BIL06}). There are several model refinements to reduce this over dissipation such as the damping function \cite{A.P}. 
 \begin{re}
Shear flow has two natural microscales. The Kolmogorov microscale is $\eta \simeq \Rey^{-\frac{3}{4}}\, L$, and describes the size of the smallest persistent motion away from walls. The second, turbulent boundary layer, length is $\Rey^{-1}\, L$. Our analysis here, and (\ref{summery}), indicates this latter scale is the more important one for shear flow. 

\end{re}

In Sections 2, we collect necessary mathematical tools. In Section 3, the major results are proven. We end this report  with numerical illustrations  and conclusions in Sections 4 and 5.

\begin{figure}[t]
\includegraphics[scale=0.17]{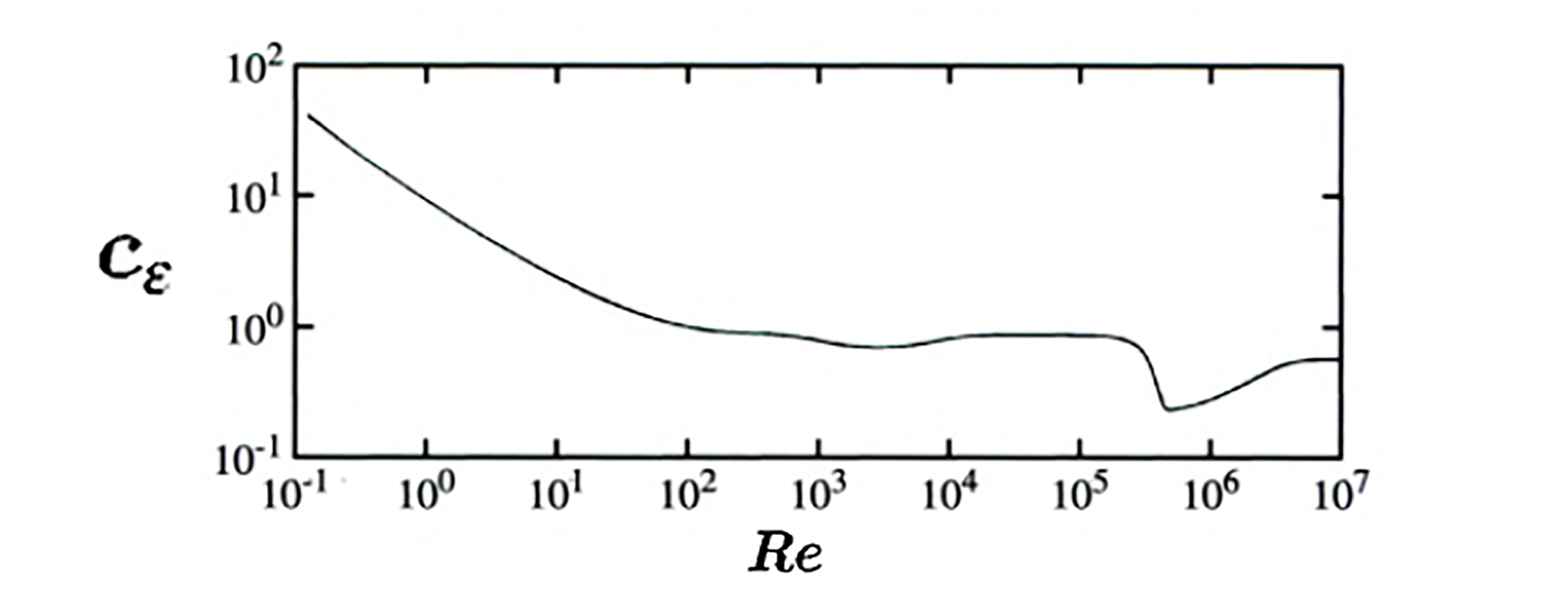}
\caption{Variation of dissipation coefficient $C_{\epsilon}$ with Reynolds number  $\Rey$ for a circular cylinder,  based on exprimental data adopted from \cite{F95}.}
\label{Tritton}
\end{figure}


\subsection{Motivation and Related Works}
  In turbulence, dissipation predominantly occurs  at small scales. Once a mesh (of size $h$) is selected, an eddy viscosity term $\grad \x ((C_s \delta)^2 |\grad u| \grad u)$  in (\ref{SM})  with $\delta = \mathcal{O}(h)$ is added to the Navier-Stokes equations (NSE), $C_s=0$ in (\ref{SM}),  to model this extra dissipation which can not be captured by the NSE viscous term on an under-resolved mesh.

  Turbulence is about prediction of velocities' averages  rather than the point-wise velocity. One commonly used average in turbulent flow modeling is time-averaging. Time averages seem to be predictable even when dynamic flow behavior over bounded time intervals is irregular \cite{F95}.  Time averaging is defined in terms of the limit superior ($\limsup$) of a function as, 
$$\langle\psi(\cdot)\rangle= \limsup\limits_{T\rightarrow\infty}  \, \frac{1}{T} \int_{0}^{T} \psi(t)\, dt.$$

 The classical \textbf{Equilibrium Dissipation Law} (Kolmogrov's conventional turbulence theory) is based on the concept of the energy cascade. Energy is input into the largest scales of the flow,  then the kinetic energy cascade from large to small scale of motions. When it reaches a scale small enough for viscous dissipation to be effective, it dissipates mostly into heat (Richardson 1922, \cite{R22}). Since viscous dissipation is negligible through this cascade, the energy dissipation rate is related then to the power input to the largest scales at the first step in the cascade. These largest eddies have energy $\frac{1}{2} U^2$ and time scale $\tau=\frac{L}{U}$, this implies the \textit{Equilibrium Dissipation Law} for time-averaged energy dissipation rate  $\langle  \varepsilon\rangle$ (Kolmogorov 1941), 

$$\langle  \varepsilon\rangle = \mathcal{O}(\frac{U^2}{\tau}) \simeq  C_{\varepsilon}\frac{U^3}{L},$$
with $C_{\varepsilon}=$ constant.  Saffman \cite{S68}, addressing the estimate of energy dissipation rates, $\langle  \varepsilon\rangle  \simeq  \frac{U^3}{L},$ and  wrote that,
 \vspace{10pt}

\textit{”This result is fundamental to an 
understanding of turbulence and yet still lacks theoretical support.”} 
\rightline{{\rm ---  P.G. Saffman 1968.}}

\vspace{10pt}

 Information about $C_{\varepsilon}$ carries interesting information about the structure of the turbulent flow. Therefore mathematically rigorous upper bounds on $C_{\varepsilon}$ are of high relevance. There are many analytical,  numerical and experimental evidence to support the equilibrium dissipation law; some of them are mentioned below.
 
\begin{figure}[b]
\includegraphics[scale=0.5]{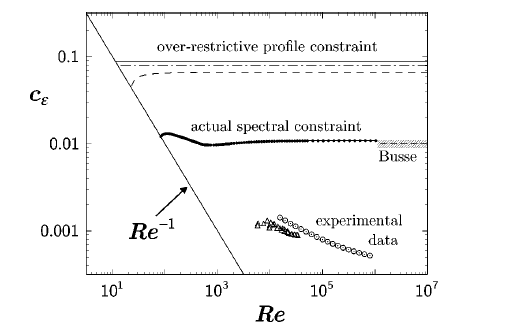}
\caption{$C_{\epsilon}$ versus $\Rey$, adopted from \cite{NGH99}.}
\label{Cepsilon}
\end{figure}

Taylor \cite{T35} considered $C_{\varepsilon}$ to be constant for geometrically similar boundaries. His work led to the question "If $C_{\varepsilon}$ depends on geometry, boundary and inlet conditions". Since then, a multitude of laboratory experiments \cite{S98} and numerical
simulations \cite{BSB07} concerned with isotropic homogeneous turbulence seem to confirm that $C_{\varepsilon}$ is independent of Reynolds number in the limit of high Reynolds number, but are not
conclusive as to whether $C_{\varepsilon}$ is universal at such high Reynolds number values. In fact, the high Reynolds number values
of $C_{\varepsilon}$ seem to differ from flow to flow. The first results in 3D obtained for turbulent shear flow  between parallel plates by Howard \cite{H72} and Busse \cite{B70}  under assumptions that the flow is statistically stationary. There is an asymptotic non-zero upper bound, $C_\epsilon \simeq 0.01$  as $\Rey\rightarrow \infty$, first derived by Busse \cite{B70} (Joining dashed line in Figure \ref{Cepsilon}). The lower bound on $C_{\varepsilon}$  is already given for laminar flow by Doering and Constantin, $ \frac{1}{\Rey} \leq C_{\varepsilon}$ (solid slanted straight line in Figure \ref{Cepsilon}).  Rigorous asymptotic $\Rey\rightarrow \infty$ dissipation rate bounds of the form $ \langle  \varepsilon\rangle \leq C_{\varepsilon}\frac{U^3}{L}$,   with  $C_\epsilon \simeq 0.088$ (topmost horizontal solid line in Figure \ref{Cepsilon}), were derived for a number of boundary-driven flows during the 1990s (Doering and  Constantin \cite{DC92}). The residual dissipation bound like this appeared for a shear layer turbulent Taylor-Couette flow-  where $L$ was the layer thickness and U was the overall velocity drop across the layer. The upper bound has been  confirmed with higher precision in \cite{NGH97} as $C_\epsilon \simeq 0.01087$ (Heavy dots in Figure \ref{Cepsilon}).
The corresponding experimental data measured by Reichardt \cite{R59}  for the plane Couette flow (Triangles in Figure \ref{Cepsilon}) and by Lathrop, Fineberg, and Swinney \cite{LFS92} for the small-gap Taylor-Couette system (Circles in Figure \ref{Cepsilon}).

\subsection{Why the  Eddy Viscosity model  and not the  Navier Stokes?}\

A shear flow is one where the boundary condition is tangential.  The flow problem  with the boundary condition like (\ref{BC})  becomes very close to the flow between rotating cylinders. Flow between rotating cylinders is one of the classic problems in experimental fluid dynamics (e.g. \cite{F95}).

Consider the Navier Stokes equations, $C_s=0$ in (\ref{SM}), with the boundary condition (\ref{BC}). The difficulty in the analysis of energy dissipation of  the shear flow appears due to the effect of the non-homogeneous boundary condition on the flow. The classical approach is required to construct a  careful and non-intuitive choice of background flow $\Phi$ known as the Hopf extension  \cite{H55}.

\begin{definition}  
\label{backgroundflow}
Let $\Phi(x,y,z) := (\phi(z),0,0)\tran$, where,
\begin{equation}
\phi(z) =
\left\{
	\begin{array}{ll}
		0  & \mbox{if } z\in [0,L-h]  \\
		\frac{U}{h} (z-L+h)  & \mbox{if }  z\in [L-h, L]
	\end{array}.
\right.
\end{equation}

\end{definition}

The strategy is to subtract off the inhomogeneous boundary conditions  (\ref{BC}).  Consider  $v= u-\Phi$, then $v$ satisfies homogeneous boundary conditions. Substituting $u=v+\Phi$ in the NSE yields, 
\begin{equation}\label{HomoEq}
  v_t+v \cdot \nabla v -\nu \Delta v + \nabla p+\phi(z) \frac{\partial v}{\partial x} + v_3 \phi'(z) (1,0,0)= \nu \phi''(z)(1,0,0),
   \end{equation}
   $$ \grad \cdot v =0.$$
The key idea in making progress in the mathematical understanding of the energy dissipation rate is the energy inequality. 
Taking the  inner product (\ref{HomoEq}) with $v=(v_1,v_2,v_3)$ and integrating over $\Omega$ gives,
\begin{equation}\label{L2-1}
\frac{1}{2} \frac{\partial}{\partial t} ||v||^2 + \nu ||\grad v||^2 + \int_{\Omega} \big(\phi(z) \frac{\partial v}{\partial x}\x v + v_1 v_3 \phi'(z)\big) \, dx =\int_{\Omega} \nu \phi''(z) v_1\, dx.
\end{equation}

However the NSE  are rewritten to make the viscous term easy to handle, the nonlinear term changes and must again be wrestled with. To calculate energy dissipation, the non-linear terms should be controlled by the diffusion term. Since both terms are quadratic in $u$, considering current tools in analysis, this control is doable only by assuming $h$ in Definition  \ref{backgroundflow} being small. Doering and Constantin \cite{DC92} used the NSE to find an upper bound on  the time-averaged energy dissipation rate  for shear driven turbulence, assuming $h \simeq \Rey^{-1}$.   Similar estimations have been proven by Marchiano \cite{M94}, Wang \cite{W97} and Kerswell \cite{K97} in more generality. For the semi-discrete NSE, John, Layton and Manica \cite{JLM07}  have shown that assuming infinite resolution near the boundaries, computed time-averaged energy dissipation rate $\langle\varepsilon(u^h) \rangle$ for the shear flow scales as predicted for the continuous flow by the Kolmogorov theory,
$$\langle\varepsilon(u^h) \rangle \,\simeq \,\frac{U^3}{L}.$$

 Now the  question arises:  How to find $<\varepsilon(u^h)>$ for shear driven turbulence without any restriction on the mesh size? With current analysis tools, the effort seems to be hopeless. Therefore  the p-laplacian term, 
\begin{equation}\label{Plaplacian}
\grad \x (|\grad u|^{p-2}\, \grad u),
\end{equation}
 was added to the NSE. This modification will lead to the  $p^{\mbox{th}}-$ degree term on the left side of the energy inequality. Then applying appropriate Young's inequality on the non-linear term, we will absorb it on the $p^{\mbox{th}}-$ degree term without any restriction on $h$.  Herein, we focus  on  $p=3$ (the Smagorinsky model (\ref{SM})), but the analysis for $p>3$ remains as an interesting problem.

\section{Mathematical preliminaries, notations and definitions }
We use the standard notations $L^p(\Omega), W^{k,p}(\Omega), H^k(\Omega)= W^{k,2}(\Omega)$ for the Lebesgue and Sobolev spaces respectively. The inner product in the space $L^2(\Omega)$ will be denoted by $(\cdot,\cdot)$ and its norm by $ || \x || $ for scalar, vector and tensor quantities. Norms in Sobolev spaces $H^k(\Omega), k>0$, are denoted by  $ || \x ||_{H^k}$ and the usual $L^p$ norm is denoted by $|| \x ||_{p}$. The symbols $C$ and $C_i$ for $i=1, 2, 3$ stand for generic positive constant independent of the $\nu$, $L$ and $U$.  In addition, $\grad u$ is the gradient tensor $(\grad u)_{ij} =\frac{\partial u_j}{\partial x_i}$ for $i, j= 1, 2 , 3$.\\

The Reynolds number is $\Rey=\frac{U L}{\nu}$. The time-averaged energy dissipation rate for model (\ref{SM})  includes dissipation due to the viscous forces and the turbulent diffusion. It is given by,

\begin{equation}
\langle  \varepsilon (u)\rangle =  \limsup\limits_{T\rightarrow\infty} \frac{1}{T} \int_{0}^{T}  (\frac{1}{|\Omega|} \int_ \Omega \nu |\nabla u|^2 + (c_s\delta)^2  |\nabla u|^3 dx) \,\, dt.
\end{equation}

\begin{definition}

The velocity at a given time $t$ is sought in the space

$\mathbb{X}(\Omega):= \{u \in H^1(\Omega):u(x,y,0)=(0,0,0)\tran, \, u(x,y,L)=(U,0,0)\tran, \, \mbox{$u$ is $L$-periodic in $x$ and $y$ direction} \}.$

The test function space is 

$\mathbb{X}_0(\Omega):= \{u \in H^1(\Omega):u(x,y,0)=(0,0,0)\tran,\, u(x,y,L)=(0,0,0)\tran,\, \mbox{$u$ is $L$-periodic in $x$ and $y$ direction} \}.$

The pressure at time $t$ is sought in 

${Q}(\Omega):= L_0^2(\Omega) = \{q \in L^2(\Omega): \hspace{3pt}\int_{\Omega} q dx= 0\}.$

And the space of divergence-free functions is denoted by

$V(\Omega):=\{ u \in \mathbb{X}(\Omega):  \hspace{3pt}(\grad \x u,q)=0  \hspace{5pt} \forall  q \in {Q}\}.$
\end{definition}

\begin{lem} \label{lemma1}
 $\Phi$ in Definition \ref{backgroundflow} satisfies, 
  \begin{tasks}(3)
\task ${\Vert\Phi\Vert}_{\infty} \leq U, $
\task ${\Vert\nabla\Phi\Vert}_{\infty} \leq\frac{U}{h},$
\task ${\Vert \Phi\Vert}^2  \leq\frac{U^2 L^2 h}{3},$  
\task $ \Vert \nabla \Phi\Vert^ 2 \leq\frac{U^2 L^2}{h},$
\task ${\Vert \nabla\Phi\Vert}_{3}^3  \leq\frac{U^3 L^2 }{h^2},$ 
\task $||\grad \Phi||_{\frac{3}{2}}^3= \frac{U^3 L^4}{h}.$
\end{tasks}
\end{lem}

\begin{proof}
They all are the immediate consequences of the Definition \ref{backgroundflow}.
\end{proof}
\begin{re}
 $\Phi$ extends the boundary conditions (1.2) to the interior of $\Omega$. Moreover, it is a divergence-free function. $h\in(0,L)$ stands for the spatial  mesh size.
\end{re}

Using a standard scaling argument in the next lemma shows how the constant $C_{\ell}$ in the Sobolev inequality $||u||_6 \leq C_{\ell}||\nabla u||_3$ depends on the geometry in $\mathbb{R}^3$.

\begin{lem}\label{lemma2}
Consider the Sobolev inequality $||u||_6 \leq C_{\ell} ||\nabla u||_3$, then $C_{\ell}$ is the order of $L^\frac{1}{2}$ when the domain is $\Omega=[0,L]^3$,  i.e.
$$||u||_6 \leq C \, L^{\frac{1}{2}} ||\nabla u||_3.$$
\end{lem}
\begin{proof}
Let $\Omega=[0,L]^3$, $\hat{\Omega}=[0,1]^3$  and for simplicity $x=(x_1,x_2,x_3)$ and $\hat{x}=(\hat{x}_1,\hat{x}_2,\hat{x}_3)$. Consider the change of variable $\eta: \hat{\Omega} \longrightarrow \Omega$ by $\eta (\hat{x})= L\,\hat{x}= x.$ If $ {u}(L \hat{x}) = \hat{u}(\hat{x})$, using the chain rule gives, 

$$
\frac{d}{d\hat{x_i}} (\hat{u}(\hat{x}))=\frac{d}{d\hat{x_i}} (u(L\hat{x}))= L \frac{d u}{dx_i}(L\hat{x}),
$$
for  $i=1,2,3$ and since $L\hat{x}=x$ we have,
\begin{equation}\label{Sobolev1}
\frac{1}{L} \frac{d \hat{u}}{d \hat{x_i}}= \frac{d u}{d x_i}.
\end{equation}
Using (\ref{Sobolev1}) and the change of variable formula, one can show, 
$$
||\frac{\partial u_i}{\partial x_j}||_3^3= \int_\Omega |\frac{\partial u_i}{\partial x_j}|^3 \, dx = \int_{\hat{\Omega}} \frac{1}{L^3} |\frac{\partial  \hat{u}_i}{\partial \hat{x}_j}|^3\, L^3 d\hat{x}= ||\frac{\partial \hat{u}_i}{\partial \hat{x}_j}||_3^3,
$$
therefore, 
 
\begin{equation}\label{Sobolev2}
||\nabla u||_3^3= \sum_{i,j=1}^{3} ||\frac{\partial u_i}{\partial x_j}||_3^3= \sum_{i,j=1}^{3} ||\frac{\partial \hat{u}_i}{\partial \hat{x}_j}||_3^3= ||\hat{\nabla} \hat{u}||_3^3.
\end{equation}
Again by applying the change of variable formula, we have,  
$$ ||u||_{L^6(\Omega)} = (\int_\Omega |u|^6 dx)^{\frac{1}{6}} = (\int_{\hat{\Omega}} L^3\, |\hat{u}|^6 d\hat{x})^{\frac{1}{6}} = L^{\frac{1}{2}} ||\hat{u}||_{L^6(\hat{\Omega})},
$$
using the Sobolev inequality and also inequality (\ref{Sobolev2}) on the above equality leads to,
\begin{equation}
 ||\hat{u}||_{L^6(\hat{\Omega})}= L^{-\frac{1}{2}} ||u||_{L^6(\Omega)} \leq L^{-\frac{1}{2}} C_{\ell}||\nabla u||_3 = L^{-\frac{1}{2}} C_{\ell} ||\hat{\nabla} \hat{u}||_3,
\end{equation}
therefore as claimed  $C_{\ell} = \mathcal{O}(L^{\frac{1}{2}})$.
\end{proof}

We will need the well-known dependence of  Poincar\'e -Friedrichs inequality constant on the domain. A straightforward argument in the thin domain $\mathcal{O}_{h}$ implies Lemma \ref{lemma3}, which is a special case of Poincar\'e -Friedrichs inequality.

 \begin{lem} \label{lemma3}
 Let $\mathcal{O}_{h}=\lbrace(x,y,z)\in \Omega : L-h \leq z \leq L\rbrace  $ 
  be the region close to the upper boundary. Then we have
  \begin{equation}
   {\Vert u - \Phi\Vert}_{L^2(\mathcal{O}_{h})} \leq \,h \,  {\Vert\nabla (u -\Phi)\Vert }_{L^2(\mathcal{O}_{h})}.
    \end{equation}
  \end{lem}
  \begin{proof}
  The proof is standard.
  \end{proof}

\subsection{Variational Formulation and Discretization}
The variational formulation is obtained by taking the scalar product $ v\in\mathbb{X}_0$ and $q \in L_0^2 $ with (\ref{SM}) and integrating over the space $\Omega$. 

\begin{equation}
  \label{Variational} 
  \begin{aligned}
 (u_t, v) + \nu (\nabla u,\nabla v) + (u \x \grad u,v) - (p,\nabla \x v)+((C_s \delta)^2 |\grad u| \grad u,\grad v)&= 0 & \forall v \in \mathbb{X}_0  ,
    \\
  (\nabla \x u,q) & =0 & \forall q\in L_0^2  ,
   \\
    (u(x,0) - u_0(x) , v)
    &=
    0 &\forall v \in \mathbb{X}_0.  
   \end{aligned}
  \end{equation}
  To discretize the SM, consider two finite-dimensional spaces $\mathbb{X}^h \subset \mathbb{X}$  containing  linears and $\mathbb{Q}^h \subset \mathbb{Q}$ satisfying the following  discrete inf-sup condition where $\beta ^h >0$ uniformly in $h$ as $h \rightarrow 0$,
\begin{equation}\label{LBBH}
\inf_{q^h \in \mathbb{Q}^h} \sup_{v^h \in X^h} \frac{(q^h, \grad \x v^h)}{||\grad v^h|| ||q^h||} \geq \beta ^h >0.
\end{equation}
The inf-sup condition (\ref{LBBH}) plays a significant role in studies of the finite-element approximation of the Navier-Stokes equations. It is usually taken as a criterion of whether or not the families of finite-element spaces yield stable approximations. In the  words, it ensures given the unique velocity, there is a corresponding pressure. It is also critical to bounding the fluid pressure and showing the pressure is stable.

Consider a subspace $V^h \subset \mathbb{X}^h$  defined by:
$$V^h:= \{ v^h \in X^h: (q^h, \grad \x v^h)=0, \forall q^h \in \mathbb{Q}^h \}.$$
Note that most often $ V^h\not\subset V$ and $\grad \x u^h \neq 0$ for any $u^h \in V^h.$  Thus we need an extension of the trilinear from $(u\x \grad v,w)$ as follows. 
\begin{definition} \label{deftrilinear}
\textbf{(Trilinear from)} Define the trilinear form  $b$ on $\mathbb{X}\times \mathbb{X} \times \mathbb{X}$ as,  
$$b(u,v,w):= \frac{1}{2}(u\x \grad v,w)- \frac{1}{2}(u\x \grad w,v).$$
\end{definition}

\begin{lem} \label{trilinear}
The nonlinear term  $b(\x,\x,\x)$ is continuous on $\mathbb{X}\times \mathbb{X} \times \mathbb{X}$ (and thus on $V \times V \times V$ as well). Moreover,  we have the following skew-symmetry properties for $b(\x,\x,\x),$ 
$$b(u,v,w)= (u\x \grad v,w) \hspace{10pt} \forall u\in V \mbox{ and } v ,w\in \mathbb{X},$$
Moreover,
$$b(u,v,v)=0 \hspace{10pt} \forall u ,v\in  \mathbb{X}.$$
\end{lem}
\begin{proof}
The proof is standard, see p.114 of Girault and Raviart \cite{R79}.
\end{proof}

The semi-discrete/continuous-in-time finite element approximation continues by selecting finite element spaces $\mathbb{X}_0^h \subset \mathbb{X}_0$. The approximate velocity and pressure of the Smagorinsky problem (\ref{SM}) are $u^h: [0,T] \longrightarrow \mathbb{X}^h$ and $p^h : (0,T] \longrightarrow \mathbb{Q}^h$ such that,

\begin{equation}
  \label{FEM-SM} 
  \begin{aligned}
 (u^h_t, v^h) + \nu (\nabla u^h,\nabla v^h) + b(u^h,u^h,v^h) - (p^h,\nabla \x v^h)+((C_s \delta)^2 |\grad u^h| \grad u^h,\grad v^h)&= 0 & \forall v^h \in \mathbb{X}^h_0  ,
    \\
  (\nabla \x u^h,q^h) & =0 & \forall q^h\in \mathbb{Q}^h  ,
   \\
    (u^h(x,0) - u_0(x) , v^h)
    &=
    0 &\forall v^h \in \mathbb{X}^h_0.  
   \end{aligned}
  \end{equation}

\section{Theorems and Proofs} 

In the following theorems, we present upper bounds on the computed time-averaged energy dissipation rate for the Smagorinsky model (\ref{SM}) subject to the shear flow boundary condition (\ref{BC}). Theorem \ref{thm1} considers the case  when the mesh size is fine enough $h < \mathcal{O}(\Rey^{-1}) \, L$.  The restriction on the  mesh size arises from the mathematical analysis of constructible  background flow in finite element space.

 On the other hand, Theorem \ref{thm2} investigates $\langle  \varepsilon (u^h)\rangle$ for any mesh size $0<h<L$. We then take a minimum of two bounds to find the optimal upper bound in (\ref{summery}) with respect to the current analysis.

\begin{thm}\label{thm1}
\textbf{(fully-resolved mesh)} Suppose $u_0 \in L^2(\Omega)$ and mesh size $h <(\frac{1}{5} \Rey^{-1}) L$. Then $\langle  \varepsilon (u^h)\rangle $ satisfies,
$$\langle  \varepsilon (u^h)\rangle \leq  C\, \big[1+  (\frac{c_s \delta}{L})^2\, \Rey ^2\big]\, \frac{U^3}{L}.$$
\end{thm}

\begin{proof}

  Let $h=\gamma L$ in Definition \ref{backgroundflow}. We shall select  $0<\gamma<1$ appropriately so that $\Phi$ belongs to the finite element space.  Take $v^h = u^h- \Phi \in \mathbb{X}^h_0$ in (\ref{FEM-SM}). Since $b(\x,\x,\x)$ is skew-symmetric (Lemma \ref{trilinear}) and $\grad \x \Phi =0$, we have, 

$$(u^h_t,u^h-\Phi)+\nu(\nabla u^h,\nabla u^h-\nabla\Phi)+b(u^h,u^h,u^h-\Phi)+((C_s \delta)^2|\nabla u^h| \nabla u^h, \nabla u^h-\nabla\Phi)=0,$$
and integrate in time to get,
\begin{equation} \label{T1-0}
\begin{split}
\frac{1}{2} \Vert u^h(T)\Vert^2 - \frac{1}{2} \Vert u^h(0)\Vert^2 &+\nu \int _{0}^{T} \Vert \nabla u^h\Vert^2 dt +\int_{0}^{T}(\int_\Omega (C_s \delta)^2  |\nabla u^h|^3 dx) dt = (u^h(T),\Phi)-(u^h(0),\Phi) \\
 & +\int _{0}^{T}b(u^h,u^h,\Phi) dt+\nu \int _{0}^{T} (\nabla u^h,\nabla\Phi) dt
 + (C_s \delta)^2 \int_{0}^{T}(|\nabla u^h| \nabla u^h, \nabla \Phi) dt.\\
\end{split}
\end{equation}

 Using  Cauchy-Schwarz-Young's inequality and Lemma \ref{lemma1} for $h=\gamma L$ to bound the right-hand side of the above energy equality.

\begin{equation}\label{T1-1}
(u^h(T),\Phi)\leq \frac{1}{2}\Vert u^h(T)\Vert^2 +\frac{1}{2}\Vert\Phi\Vert^2 = \frac{1}{2}\Vert u^h(T)\Vert^2 +\frac{U ^2 \gamma L^3}{6}.
\end{equation}
\begin{equation}\label{T1-2}
(u^h(0),\Phi)\leq \Vert u^h(0)\Vert \Vert\Phi\Vert = \sqrt{\frac{\gamma}{3}} U L^{\frac{3}{2}} \Vert u^h(0)\Vert.
\end{equation}
\begin{equation}\label{T1-3}
\nu \int _{0}^{T} (\nabla u^h,\nabla\Phi) dt \leq \frac{\nu}{2} \int _{0}^{T}\Vert \nabla u^h\Vert^2 + \Vert \nabla\Phi\Vert ^2  dt = \frac{\nu}{2} \int _{0}^{T}\Vert \nabla u^h\Vert^2 dt + \frac{\nu}{2}  \frac{U^2 L}{\gamma}T.
\end{equation}

 Next, the nonlinear term $b(\x,\x,\x)$  in (\ref{T1-0}) can be rewritten as, 
\begin{equation} \label{NLeq}
\begin{split}
b(u^h,u^h,\Phi) & = b(u^h-\Phi,u^h-\Phi,\Phi)+ b(\Phi,u^h-\Phi,\Phi)  \\
 & = \frac{1}{2}  b(u^h-\Phi,u^h-\Phi,\Phi) - \frac{1}{2} b(u^h-\Phi,\Phi,u^h-\Phi)\\
 & +\frac{1}{2}  b(\Phi,u^h-\Phi,\Phi) - \frac{1}{2}  b(\Phi,\Phi,u^h-\Phi).
\end{split}
\end{equation}
Each term in (\ref{NLeq}) is estimated separately as follows using Lemma \ref{lemma3} and  Cauchy-Schwarz-Young's inequality. We will also take the advantage  of the fact that  $b(u^h,u^h,\Phi)$ is an integration on $\mathcal{O}_{\gamma L}$ since  supp$(\Phi) =\overline{\mathcal{O}} _{\gamma L}$. For the first term in (\ref{NLeq}) we have, 
\begin{equation} \label{NLeq1}
\begin{aligned}
& b(u^h-\Phi,u^h-\Phi,\Phi) \leq \Vert\Phi\Vert_{L^\infty} \Vert u^h-\Phi\Vert  \Vert \nabla(u^h-\Phi)\Vert \leq\gamma LU \Vert \nabla (u^h-\Phi)\Vert^2_{L^2}\\
  & \leq\gamma LU \Vert \nabla u^h- \nabla\Phi\Vert^2 \leq UL\gamma (\Vert\nabla u^h\Vert+ \Vert\nabla\Phi\Vert)^2  \leq UL\gamma (2 \Vert \nabla u ^h\Vert^2 + 2 \Vert \nabla \Phi \Vert^2) \\
 &\leq UL\gamma (2 \Vert \nabla u^h \Vert^2+2 \frac{U^2 L}{\gamma})= 2 UL\gamma \Vert \nabla u^h\Vert^2 +2 U^3L^2.\\
 \end{aligned}
\end{equation}
For the second term we have,
\begin{equation} \label{NLeq2}
\begin{aligned}
&b(u^h-\Phi,\Phi,u^h-\Phi)  \leq \Vert \nabla\Phi\Vert_{L_\infty} \Vert u^h-\Phi\Vert ^2 \leq \frac{U}{\gamma L} \gamma ^2 L^2 \Vert \nabla(u^h-\Phi)\Vert^2\\
& \leq \gamma ^2 L^2 \frac{U}{\gamma L}(2 \Vert \nabla u^h \Vert^2+2 \frac{U^2 L}{\gamma})=2 \gamma LU \Vert \nabla u^h\Vert^2+ 2U^3 L^2.
\end{aligned}
\end{equation} 
The third one is estimated as,
\begin{equation} \label{NLeq3}
\begin{split}
&b(\Phi,u^h-\Phi,\Phi)  \leq \Vert \Phi\Vert_{L^\infty} \Vert \nabla (u^h-\Phi)\Vert \Vert \Phi\Vert \leq U \sqrt{\frac{U^2\gamma L^3}{3}}(\Vert \nabla u^h\Vert+ \Vert \nabla \Phi\Vert)\\
& \leq U \sqrt{\frac{U^2\gamma L^3}{3}}(\Vert \nabla u^h\Vert+ \sqrt{\frac{U^2 L}{\gamma}})  \leq\frac{U^2\gamma ^ {\frac{1}{2}} L^{\frac{3}{2}}}{\sqrt{3}} \Vert \nabla u^h\Vert + \frac{U^3 L^2}{\sqrt{3}}\\
 & =[ \frac{U^{\frac{3}{2}} L}{\sqrt{3}}]\,\,\, [(U\gamma L)^{\frac{1}{2}}\Vert \nabla u^h\Vert] + \frac{U^3 L^2}{\sqrt{3}} \leq (\frac{U^3 L^2}{6}) + \frac{1}{2} UL\gamma \Vert \nabla u^h\Vert ^2+ \frac{U^3 L^2}{\sqrt{3}} \\
 &= \frac{1}{2} UL\gamma \Vert \nabla u^h\Vert^2 + (\frac{\sqrt{3}}{3}+\frac{1}{6})   U^3 L^2. 
 \end{split}
\end{equation}
And finally the last one satisfies, 
\begin{equation} \label{NLeq4}
\begin{split}
&b(\Phi,\Phi,u^h-\Phi) \leq \Vert\Phi\Vert_{L^\infty} \Vert \nabla\Phi\Vert \Vert u^h-\Phi\Vert  \leq U\sqrt{\frac{U^2 L}{\gamma}} \gamma L \Vert \nabla (u^h-\Phi)\Vert\\
&\leq U^2 \gamma^{\frac{1}{2}} L^{\frac{3}{2}} (\Vert\nabla u^h\Vert +\Vert\nabla\Phi\Vert)\leq U^2 \gamma^{\frac{1}{2}} L^{\frac{3}{2}} (\Vert\nabla u^h\Vert +(\frac{U^2 L}{\gamma})^{\frac{1}{2}})\\
&= U^2 \gamma^{\frac{1}{2}} L^{\frac{3}{2}} \Vert\nabla u^h\Vert + U^3 L^2 =[U^{\frac{3}{2}}L]\,\,[(UL\gamma)^{\frac{1}{2}} \Vert \nabla u^h\Vert]+ U^3 L^2\\
&\leq\frac{1}{2}(U^3 L^2)+ \frac{1}{2} UL\gamma \Vert \nabla u^h\Vert^2 + U^3 L^2 =\frac{1}{2} UL\gamma \Vert \nabla u^h\Vert^2+\frac{3}{2}U^3 L^2.\\
\end{split}
\end{equation}
Use (\ref{NLeq1}), (\ref{NLeq2}), (\ref{NLeq3}) and (\ref{NLeq4}) in (\ref{NLeq}) gives the final estimation for the non-linear term as below.

\begin{equation}\label{T1-4}
 | b(u^h,u^h,\Phi)|\leq \frac{5}{2}UL\gamma \Vert \nabla u^h\Vert^2 +\frac{19}{6} U^3 L^2.
\end{equation}

 Finally,  using  Hölder's  and Young's inequality for $p=\frac{3}{2}$ and $q=3$ and  lemma \ref{lemma2}  on the last term  gives,


\begin{equation}\label{T1-5}
\begin{split}
|(|\nabla u^h| \nabla u^h, \nabla \Phi)| & \leq \int_\Omega |\nabla u^h|^2  \cdot \nabla \Phi \,dx\\
 & \leq (\int_\Omega |\nabla u^h|^3)^\frac{2}{3} (\int_\Omega |\nabla \Phi|^3)^\frac{1}{3}\\
 & \leq \frac{2}{3}(\int_\Omega |\nabla u^h|^3)+ \frac{1}{3} (\int_\Omega |\nabla \Phi|^3 )\leq  \frac{2}{3}(\int_\Omega |\nabla u^h|^3)+ \frac{1}{3} \frac{U^3}{\gamma ^2}.
\end{split} 
\end{equation}

Inserting (\ref{T1-1}), (\ref{T1-2}), (\ref{T1-3}), (\ref{T1-4}) and (\ref{T1-5}) in (\ref{T1-0}) implies, 
\begin{equation}\label{T1-6}
\begin{split}
(\frac{\nu}{2}- \frac{5}{2} \gamma LU) &\int_{0}^{T} \Vert \nabla u^h\Vert ^2 dt +\frac{1}{3}  \int _{0}^{T} (\int _\Omega (C_s \delta)^2 |\nabla u^h|^3 dx)dt \leq  \frac{1}{2} \Vert u^h(0)\Vert ^2+ \frac{1}{6} U^2\gamma L^3   \\
 &   + \sqrt{\frac{\gamma}{3}} U L^{\frac{3}{2}} \Vert u^h(0)\Vert  +\frac{19}{6}U^3 L^2 T + \frac{\nu}{2 \gamma} L U^2 T + \frac{1}{3} (C_s \delta)^2 \frac{U^3 T}{\gamma ^2}.
\end{split}
\end{equation}
Dividing  (\ref{T1-6}) by $T$ and $ |\Omega| = L^3$ and  taking limsup as $T \rightarrow \infty$ leads to,
 
 $$ min \lbrace \frac{1}{2} - \frac{5}{2} \frac{\gamma L U}{\nu}, \frac{1}{3}\rbrace \langle  \varepsilon(u^h)\rangle  \leq \frac{19}{6} \frac{U^3}{L} +\frac{\nu}{2\gamma} \frac{U^2}{L^2} + \frac{1}{3}(C_s \delta)^2 \frac{U^3}{\gamma ^2 L^3}.$$
Let $\gamma< \frac{1}{5} \Rey^{-1}$, then $min \lbrace \frac{1}{2} - \frac{5}{2} \frac{\gamma L U}{\nu}, \frac{1}{3}\rbrace > 0$ and the above estimate becomes,
$$\langle\varepsilon(u^h)\rangle  \leq C \,\big[1 +  (\frac{C_s \delta}{L})^2  \Rey ^2\big]\, \frac{U^3}{L}, $$
which proves the theorem.
\end{proof}

\begin{re}
The kinetic energy, $\frac{1}{2}\Vert u^h\Vert^2$, is not required in the proof to be uniformly bounded in time since it was canceled out from both sides after inserting (\ref{T1-1}) in (\ref{T1-0}). 
\end{re}

The estimate in Theorem \ref{thm1} goes to $\frac{U^3}{L}$ for fixed $\Rey$ as $C_s\,\delta\rightarrow 0$, which is consistent with the rate proven for NSE by Doering and Constantin \cite{DC92}. But it over dissipates  for fixed  $\delta$ as $\Rey\rightarrow \infty$ as derived for the continuous case by Layton  \cite{L02}. To fix this issue, one can suggest a super fine filter size $\delta \simeq \frac{1}{\Rey}$  which is not practical due to the computation cost. From here, we are motivated  to study the following  under-resolved case.

\begin{thm}\label{thm2}
\textbf{(under-resolved mesh)}  Suppose $u_0 \in L^2(\Omega)$. Then for any given mesh size $ 0< h<L$, $\langle  \varepsilon (u^h)\rangle$  satisfies, 
$$
\langle  \varepsilon (u^h)\rangle \leq C\, \bigg[  \frac{1}{\Rey} \,\frac{L}{h}+   (\frac{C_s\, \delta}{h})^2+ \frac{L^5}{(C_s \delta)^4\,h}+\frac{L^{\frac{5}{2}}}{(C_s\, \delta)^{4}}\, {h^{\frac{3}{2}}}   \bigg] \, \frac{U^3}{L}.
$$
\end{thm}

\begin{proof}
The proof is very similar to the one of Theorem \ref{thm1} except the estimation on the nonlinear term. Let  $h \in (0,L)$ be fixed from the  beginning. The strategy is to subtract off the inhomogeneous boundary conditions  (\ref{BC}).  The proof arises as well by taking $v^h=u^h-\Phi$ in the finite element problem (\ref{FEM-SM}). Then after integrating with respect to time, we get (\ref{T1-0}). The proof continues by estimating each term on the right-hand side of (\ref{T1-0}). Using Lemma \ref{lemma1} and the Cauchy-Schwarz-Young's inequality, the first three terms on the RHS of the energy equality (\ref{T1-0}) can be estimated as, 

\begin{equation}\label{T2-1}
(u^h(T),\Phi)\leq \frac{1}{2}\Vert u^h(T)\Vert^2 +\frac{1}{2}\Vert\Phi\Vert^2 = \frac{1}{2}\Vert u^h(T)\Vert^2 +\frac{1}{6} L^2 U^2 h.
\end{equation}
\begin{equation}\label{T2-2}
(u^h(0),\Phi)\leq \Vert u^h(0)\Vert \Vert\Phi\Vert = \sqrt{\frac{h}{3}} U L \Vert u^h(0)\Vert.
\end{equation}
\begin{equation}\label{T2-3}
\nu \int _{0}^{T} (\nabla u^h,\nabla\Phi) dt \leq \frac{\nu}{2} \int _{0}^{T}\Vert \nabla u^h\Vert^2 + \Vert \nabla\Phi\Vert ^2  dt = \frac{\nu}{2} \int _{0}^{T}\Vert \nabla u^h\Vert^2 dt + \frac{\nu}{2} \frac{U^2 L^2}{h} T.
\end{equation}
Next applying the Young's inequality,
$$ab \leq \frac{1}{p} a^p + \frac{1}{q} b^q,$$  
for conjugate $p=\frac{3}{2}$ and $q=3$  on $a=|\nabla u^h|^2$  and $b=|\nabla \Phi|$  gives, 

\begin{equation}\label{T2-4}
|\int_{\Omega} |\nabla u^h| \, \nabla u^h \, \nabla \Phi \, dx| \leq  \int_{\Omega} |\nabla u^h|^2 \, |\nabla \Phi| \, dx \leq \int_{\Omega} \big(\frac{2}{3}  |\nabla u^h|^3 +\frac{1}{3} |\nabla \Phi|^3\big) \, dx \leq  \frac{2}{3} \int_{\Omega}  |\nabla u^h|^3 dx + \frac{1}{3}\frac{U^3 L^2}{h^2}.
\end{equation}
Inserting (\ref{T2-1}), (\ref{T2-2}), (\ref{T2-3}) and (\ref{T2-4}) in (\ref{T1-0}) implies,

\begin{equation}\label{T2-5}
\begin{split}
 \frac{1}{2} \Vert u^h(T)\Vert ^2 -& \frac{1}{2} \Vert u^h(0)\Vert ^2 +\nu \int_{0}^{T} \Vert \nabla u^h\Vert ^2 dt +(C_s \delta)^2 \int _{0}^{T} (\int _\Omega  |\nabla u^h|^3 dx)dt  \leq  \frac{1}{2}\Vert u^h(T)\Vert^2 +\frac{1}{6}L^2 U^2 h \\
 &+ \sqrt{\frac{h}{3}} U L \Vert u^h(0)\Vert  + \frac{\nu}{2} \int _{0}^{T}\Vert \nabla u^h\Vert^2 dt + \frac{\nu}{2}  \frac{U^2 L^2}{h} T + \int_{0}^{T}b(u^h,u^h,\Phi) dt \\
 & + \frac{2}{3} (C_s \delta)^2  \int_{0}^{T}(\int_\Omega  |\nabla u^h|^3 dx)dt + \frac{1}{3} (C_s \delta)^2 \frac{U^3 L^2}{h^2}T.
\end{split}
\end{equation}\\

Finally the nonlinear term $b(u^h,u^h,\Phi)$ is estimated as follows. First from Definition (\ref{deftrilinear}), we have,

\begin{equation}\label{NL}
|b(u^h,u^h,\Phi)| \leq \frac{1}{2}| (u^h \x \grad u^h, \Phi) |+ \frac{1}{2} |(u^h \x \grad \Phi, u^h)|. 
\end{equation}
The first term on (\ref{NL}) can be estimated using  H{\"o}lder's inequality\footnote{$\int_\Omega |f\, g\, h| dx \leq ||f||_p\, ||g||_q\, ||h||_r$ for $\frac{1}{p}+\frac{1}{q}+\frac{1}{r}=1.$ } for  $ p=3, q=6$ and $r=2$,

$$| (u^h \x \grad u^h, \Phi) | \leq \int_\Omega |u^h\, \nabla u^h\, \Phi  | dx \leq ||\nabla u^h||_3\, ||u^h||_6\, ||\Phi||_2,$$
 then applying the following Young's inequality \footnote{More generally, for conjugate ($\frac{1}{p}+\frac{1}{q}=1$) exponents $a\geq 0, b\geq 0 : ab\leq \frac{\epsilon}{p} a^p+ \frac{\epsilon ^{-\frac{q}{p}}}{q} b^{q} $ for any $\epsilon \geq 0.$} for conjugate exponents $p=3$ and $q=\frac{3}{2}$,
 \begin{equation}\label{Young}
 ab\leq \frac{\epsilon}{3} a^3+ \frac{\epsilon ^{-\frac{1}{2}}}{\frac{3}{2}} b^{\frac{3}{2}},
  \end{equation}
 when $a=||\nabla u^h||_3$, $b=||u^h||_6\, ||\Phi||_2$ and $\epsilon= \frac{(C_s \delta)^2}{4}$ leads to,

$$|b(u^h,u^h,\Phi)| \leq  \frac{1}{12}(C_s \delta)^2  ||\nabla u^h||_3 ^3+ \frac{2}{3}(\frac{(C_s \delta)^2}{4})^{-\frac{1}{2}}\, ||\Phi||_2^{\frac{3}{2}}\, ||u^h||_6^{\frac{3}{2}}.$$
From Lemma \ref{lemma2} we have   $||u^h||_6 \leq C \, L^{\frac{1}{2}} ||\nabla u^h||_3$, it follows that 

$$|b(u^h,u^h,\Phi)| \leq  \frac{1}{12}(C_s \delta)^2  ||\nabla u^h||_3 ^3+ \frac{2}{3}(\frac{(C_s \delta)^2}{4})^{-\frac{1}{2}}\, ||\Phi||_2^{\frac{3}{2}}\, C {L}^{\frac{3}{4}}||\nabla u^h||_3^{\frac{3}{2}}.$$
Again apply the general Young's inequality for conjugate exponents $p=2$ and $q=2$ to the second term of the above inequality when $\epsilon = \frac{(C_s \delta)^2}{6}$, 
$$| (u^h \x \grad u^h, \Phi) | \leq  \frac{1}{6}(C_s \delta)^2  ||\nabla u^h||_3 ^3 + \frac{1}{12} (C_s \delta)^2  ||\nabla u^h||_3 ^3+  C \frac{8}{3} {L}^{\frac{3}{2}} (C_s \delta)^{-4} ||\Phi||_2^3. $$
Use $||\Phi||_2^3=(\frac{h}{3})^{\frac{3}{2}} L^3 U^3$ from the Lemma \ref{lemma1} on the above inequality and then,
\begin{equation}\label{NL-1}
| (u^h \x \grad u^h, \Phi) | \leq  \frac{1}{6}(C_s \delta)^2  ||\nabla u^h||_3 ^3 + C\, {L}^{\frac{3}{2}} (C_s \delta)^{-4} h^{\frac{3}{2}} L^3 U^3. \\
\end{equation}
The second term $|(u^h \x \grad \Phi, u^h)|$ on (\ref{NL}) can be bounded first using  H{\"o}lder's inequality for $p=6, q=\frac{3}{2} \mbox{ and } r=6,$
 
 $$| (u^h \x \grad \Phi , u^h ) | \leq \int_\Omega |u^h\, \nabla \Phi\, u^h  | dx \leq ||\nabla \Phi||_{\frac{3}{2}}\, ||u^h||^2_6.$$
Since $||u^h||_6 \leq  \, L^{\frac{1}{2}} ||\nabla u^h||_3$, Lemma \ref{lemma2}, we have,

 $$| (u^h \x \grad \Phi , u^h ) | \leq \int_\Omega |u^h\, \nabla \Phi\, u^h  | dx \leq ||\nabla \Phi||_{\frac{3}{2}}\, L\, ||\grad u^h||^2_3,$$
then use  Young's  inequality (\ref{Young}) with $a= ||\grad u^h||^2_3$ and $b= ||\nabla \Phi||_{\frac{3}{2}}\, L$ for $p=\frac{3}{2}, q=3$ and $\epsilon= \frac{3}{16} (C_s \delta)^2,$
 $$| (u^h \x \grad \Phi , u^h ) | \leq \frac{(C_s \delta)^2}{8} ||\grad u^h||^3_3+ \frac{16}{27} \frac{1}{(C_s \delta)^4}  ||\nabla \Phi||^3_{\frac{3}{2}}\, L^3.$$
Since $||\grad \Phi||_{\frac{3}{2}}^3= \frac{U^3 L^4}{h}$, Lemma   \ref{lemma1}, the above inequality turns to be,

\begin{equation}\label{NL-2}
| (u^h \x \grad \Phi , u^h ) | \leq \frac{(C_s \delta)^2}{8} ||\grad u^h||^3_3 + \frac{16}{27} \frac{U^3 \, L^7}{(C_s \delta)^4\, h},
\end{equation}
 combining (\ref{NL-1}) and (\ref{NL-2}) in (\ref{NL}), we have the following estimate on the non-linearity,
 \begin{equation}\label{NL-Bound}
|b(u^h,u^h,\Phi)| \leq \frac{7}{24}(C_s \delta)^2 ||\grad u^h||^3_3 +  \frac{U^3 \, L^7}{(C_s \delta)^4\, h}+ \frac{{L}^{\frac{3}{2}}  h^{\frac{3}{2}} L^3 U^3}{(C_s \delta)^{4}} .
 \end{equation}
 Inserting (\ref{NL-Bound}) in  (\ref{T2-5}) yields,

\begin{equation}\label{T2-7}
\begin{split}
 &\frac{1}{2} \int_{0}^{T} \nu \Vert \nabla u^h\Vert ^2 dt +\frac{1}{2} \int _{0}^{T} (\int _\Omega (C_s \delta)^2 |\nabla u^h|^3 dx)dt  \leq  \frac{1}{2} \Vert u^h(0)\Vert ^2+\frac{1}{6}L^2 U^2 h \\
 &+ \sqrt{\frac{h}{3}} U L \Vert u^h(0)\Vert  +  \frac{\nu}{2}  \frac{U^2 L^2}{h} T   + \frac{1}{3} (C_s \delta)^2 \frac{U^3 L^2}{h^2}T  + \frac{U^3 \, L^7}{(C_s \delta)^4\, h}\,T+ \frac{{L}^{\frac{3}{2}}  h^{\frac{3}{2}} L^3 U^3}{(C_s \delta)^{4}}\,T.\\
\end{split}
\end{equation}

Note that the above inequality can justify the fact that  the computed time-averaged of the energy dissipation of the solution (\ref{SM}) is uniformly bounded and hence $\langle  \varepsilon (u^h)\rangle$ is well-defined.

Dividing  both sides of the inequality (\ref{T2-7}) by $|\Omega|=L^3$ and $T$, taking lim Sup as $T\longrightarrow \infty$ leads to,
$$
\frac{1}{2} \langle  \varepsilon (u^h)\rangle \leq \frac{1}{2}\, \frac{\nu U^2}{ h L} +\frac{1}{3} (C_s \delta)^2 \frac{U^3}{h^2 L} + \frac{L^4\, U^3}{(C_s \delta)^4\, h }+ \frac{h^{\frac{3}{2}}\, L^{\frac{3}{2}}  \,U^3}{(C_s \delta)^4}, 
$$
which can be written as 
\begin{equation}\label{T2-8}
\langle  \varepsilon (u^h)\rangle \leq C\, \bigg[  \frac{1}{\Rey} \,\frac{L}{h}+   (\frac{C_s\, \delta}{h})^2+ \frac{L^5}{(C_s \delta)^4\,h}+\frac{L^{\frac{5}{2}}}{(C_s\, \delta)^{4}}\, {h^{\frac{3}{2}}}   \bigg] \, \frac{U^3}{L}.
\end{equation}
And the theorem is proved.
\end{proof}

\begin{re}
The estimate in Theorem \ref{thm2} is for a coarse mesh  $h\geq \mathcal{O}(\Rey^{-1}) \, L$. For the fully-resolved case $h\rightarrow \Rey^{-1}$ (or equivalently $C_s \delta \rightarrow 0$), the other estimate in Theorem \ref{thm1} takes over. More over, the estimate is  independent of  the viscosity at high Reynolds number. It is also dimensionally consistent. 
\end{re}

The value of $\langle  \varepsilon (u^h)\rangle $ in  Theorem \ref{thm2}   depends on three discretization parameters; the turbulence resolution length scale $\delta$, the Smagorinsky constant $C_s$ and the numerical resolution $h$. These affect $\langle  \varepsilon (u^h)\rangle $, but  have no impact on the underlying velocity field $u(x,t)$, and hence they have no effect upon $\langle  \varepsilon (u)\rangle $. 

Consider the upper bound on $\frac{\langle  \varepsilon (u^h)\rangle}{U^3/L}$  in (\ref{T2-8}) as function of  $(C_s \delta)$ and $h$, when $\Rey \gg 1$ and $L$ are being fixed. The expression includes three main terms.  The terms in the estimate are associated with physical effects as follows, 
\begin{enumerate}
\item $\lambda_1(h,C_s\delta)= \frac{1}{\Rey} \,\frac{L}{h} \hspace{73pt} \Longrightarrow \hspace{10pt} \mbox{Viscosity},$
\vspace{3pt}
\item  $\lambda_2(h,C_s\delta)=(\frac{C_s\delta}{h})^2\hspace{67pt} \Longrightarrow \hspace{10pt}\mbox{Model Viscosity},$
\vspace{3pt}
\item $\lambda_3(h, C_s\delta)= \frac{L^5}{(C_s \delta)^4\,h}+\frac{L^{\frac{5}{2}}}{(C_s\, \delta)^{4}}\, {h^{\frac{3}{2}}} \hspace{10pt}\Longrightarrow \hspace{10pt}\mbox{Non-linearity}.$
\end{enumerate}
Tracking back the proof of Theorem \ref{thm2}, $\lambda_1, \lambda_2$ and $\lambda_3$ correspond to the viscosity term $\nu \Delta u $, model viscosity $\grad \x((C_s \delta)^2 |\grad u| \grad u)$ and non-linear term $u\x \grad u$ respectively. 
 The three level sets of $\lambda_1=\lambda_2, \lambda_2=\lambda_3$ and $\lambda_1=\lambda_3$ are respectively denoted by the curves $\zeta_1, \zeta_2$ and $\zeta_3$ in  Figure \ref{3graphes}, which are calculated as,

 \begin{enumerate}
 \item $\zeta_1: \lambda_1=\lambda_2 \Longrightarrow C_s \delta = \Rey^{-\frac{1}{2}} \,  L\, (\frac{h}{L})^{\frac{1}{2}},$
\vspace{3pt}
\item $\zeta_2: \lambda_2=\lambda_3 \Longrightarrow C_s  \delta = L \,  \big[(\frac{h}{L})^{\frac{7}{2}}+ \frac{h}{L} \big]^{\frac{1}{6}},$
\vspace{3pt}
\item $\zeta_2: \lambda_1=\lambda_3 \Longrightarrow C_s  \delta = \Rey^{\frac{1}{4}} L \, \big[1+(\frac{h}{L})^{\frac{5}{2}} \big]^{\frac{1}{4}}.$
 
 \end{enumerate}

\begin{re}
Consider the horizontal axis to be $\frac{h}{L}$ and the vertical one to be $C_s \delta$.  The three level sets divide the $(\frac{h}{L})  \, (C_s \delta)$ - plane  into  four regions. The four regions $\RN{1}, \RN{2}, \RN{3}$ and $\RN{4}$ are identified with respect to the comparative magnitude size of $\lambda_1, \lambda_2$ and $\lambda_3$ on the $(\frac{h}{L})  \, (C_s \delta)$ - plane, Figure \ref{3graphes}. After comparing the magnitude of these three functions on each separate region, it can be seen that below the curve $\zeta_2$ (regions $\RN{1}$ and $\RN{2}$) the  effect of non-linearity term $ u \x \grad u$ on energy dissipation, corresponding to $\lambda_3$, dominates the  other terms. But above the curve $\zeta_2$ (regions $\RN{3}$ and $\RN{4}$) the model viscosity term, corresponding to $\lambda_2$, dominates. Surprisingly, the viscosity term $\nu \Delta u$ is never bigger than the other two terms for any choice of $h$, $\delta$ and $C_s$.
\end{re}

\begin{figure}
\includegraphics[scale=0.7]{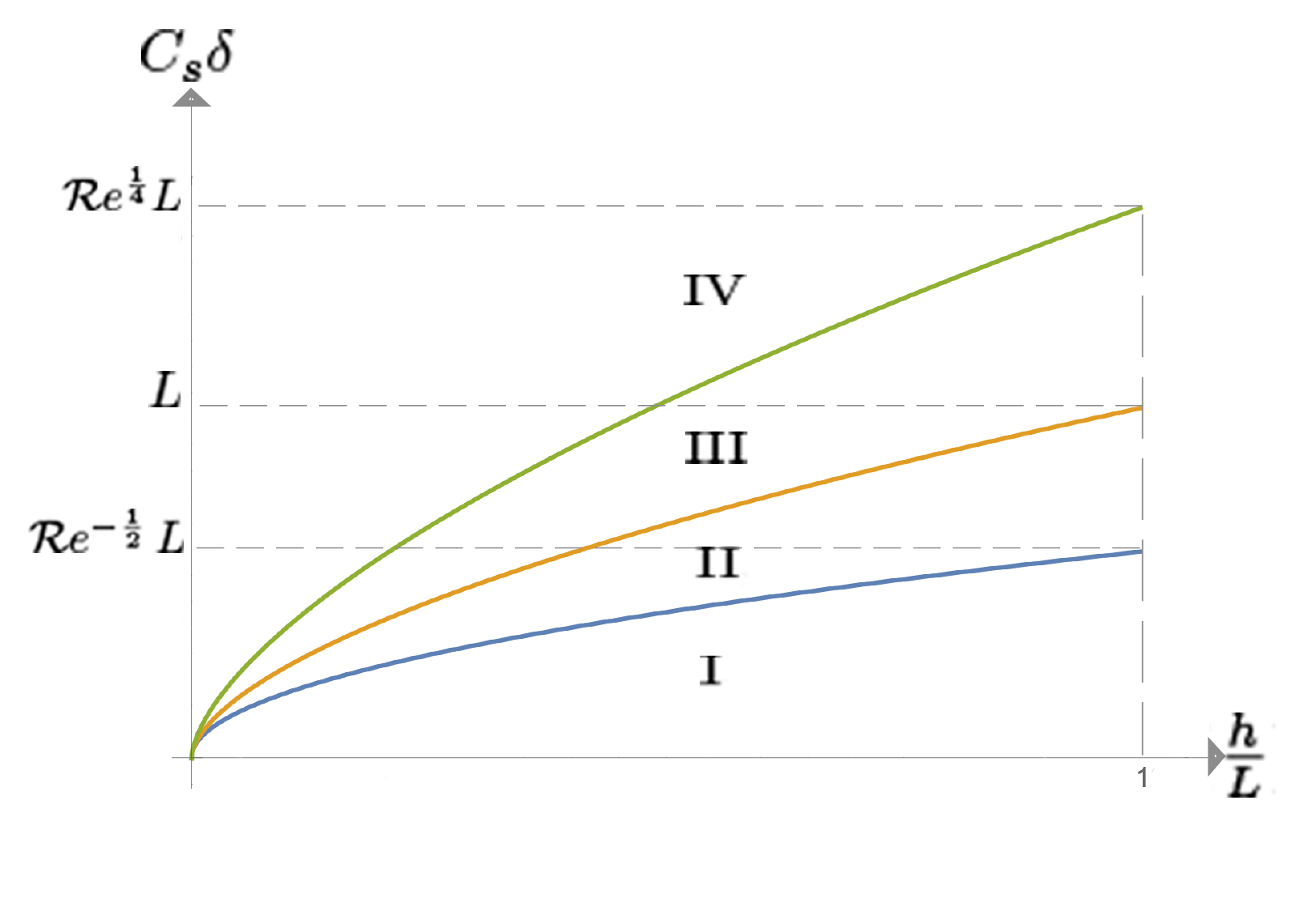}
\caption{Level Sets}
\label{3graphes}
\end{figure} 

 The Smagorinsky coefficient $C_s$ can be calibrated for a given class of flows. Its value varies from flow to flow and from domain to domain (see e.g. Page 23 of Galperin and Orszag \cite{GO93} who quote a range $0.000744 \leq C_s \leq 0.020$). In the next  corollary the optimal value, with respect the current analysis, of $C_s$  in the Theorem \ref{thm2} is investigated,  considering $\delta = h$ which is a common choice \cite{P04}.

\begin{cor}\label{Cor1}
Let $\delta = h \geq \mathcal{O} (\Rey^{-1})\, L$ be fixed. Then for any choice of $C_s > 0$, the upper estimate on $ \frac{\langle  \varepsilon (u^h)\rangle }{U^3/L} $ in the Theorem \ref{thm2} is larger than the  dissipation coefficient $C_{\epsilon}$ in Figure \ref{Cepsilon}.  
\end{cor}
\begin{proof}
Letting $\delta = h$, Theorem \ref{thm2} suggests,

\begin{equation}\label{Cor1-1}
 \frac{\langle  \varepsilon (u^h)\rangle }{U^3/L} \simeq \frac{1}{\Rey} \,\frac{L}{h}+   {C_s}^2+ \frac{1}{C_s^4} \big[ (\frac{L}{h})^5+ (\frac{L}{h})^{\frac{5}{2}}\big].
\end{equation}
Solving the minimization problem $(C_s)_{\min}= \operatorname*{arg\,min}_{C_s } F(C_s)$, the minimum of the function, 
\begin{equation}\label{Cor1-2}
 F(C_s )= \frac{1}{\Rey} \,\frac{L}{h}+   {C_s}^2+ \frac{1}{C_s^4} \big[ (\frac{L}{h})^5+ (\frac{L}{h})^{\frac{5}{2}}\big],
 \end{equation}
occurs at 
\begin{equation}\label{Cor1-3}
F_{\min}=\frac{1}{\Rey} \,\frac{L}{h}+ 2 \big[(\frac{L}{h})^5+ (\frac{L}{h})^{\frac{5}{2}}\big]^{\frac{1}{3}},
\end{equation}
for $(C_s)_{\min} = \sqrt[6]{2}\, [(\frac{L}{h})^5+(\frac{L}{h})^{\frac{5}{2}}]^\frac{1}{6}$, assuming $\mathcal{O}(\Rey^{-1})<\frac{h}{L}<1$ is fixed. This minimum amount  (\ref{Cor1-3}) dramatically exceeds the range of $C_{\epsilon}$ in Figure \ref{Cepsilon} for any typical choice of $h$ in LES.
\end{proof}

\begin{re}
Corollary \ref{Cor1} suggests  over-dissipation of  $\langle  \varepsilon (u^h)\rangle$    for any choice  of $C_s >0$. As an example, let $\frac{h}{L}=0.01$, then $F_{\min} \simeq 2000$  in (\ref{Cor1-3}) as $\Rey \rightarrow \infty$. It is much larger than the experimental range of  $C_s$ in Figure \ref{Cepsilon}, $\frac{1}{\Rey}\leq C_{\epsilon} \leq 0.1$, see Figure \ref{Campare}.

\end{re}

 However $\delta$ is taken to be $\mathcal{O}(h)$ in many literature \cite{P04},  the relationship between the grid size  $h$ and the filter size $\delta$ has been made based on heuristic instead of a sound numerical analysis (p.26 of \cite{BIL06}). Therefore for an LES of fixed computational cost (i.e. fixed $h$) and fixed length scale  $L$,  one can ask how  the artificial parameter $ C_s \,\delta$ should be selected such that the statistics of the model be consistent with numerical and experimental evidence summarized in Figure \ref{Cepsilon}. To answer this question, the inequality (\ref{T2-8}) takes over in the next corollary for under-resolved spatial  mesh $h \geq \mathcal{O}({\Rey}^{-1})\, L$,  since no model is used when (at much greater cost) a DNS is performed with $\frac{h}{L}\simeq \Rey^{-1}$.

\begin{figure}[t]
\includegraphics[scale=0.35]{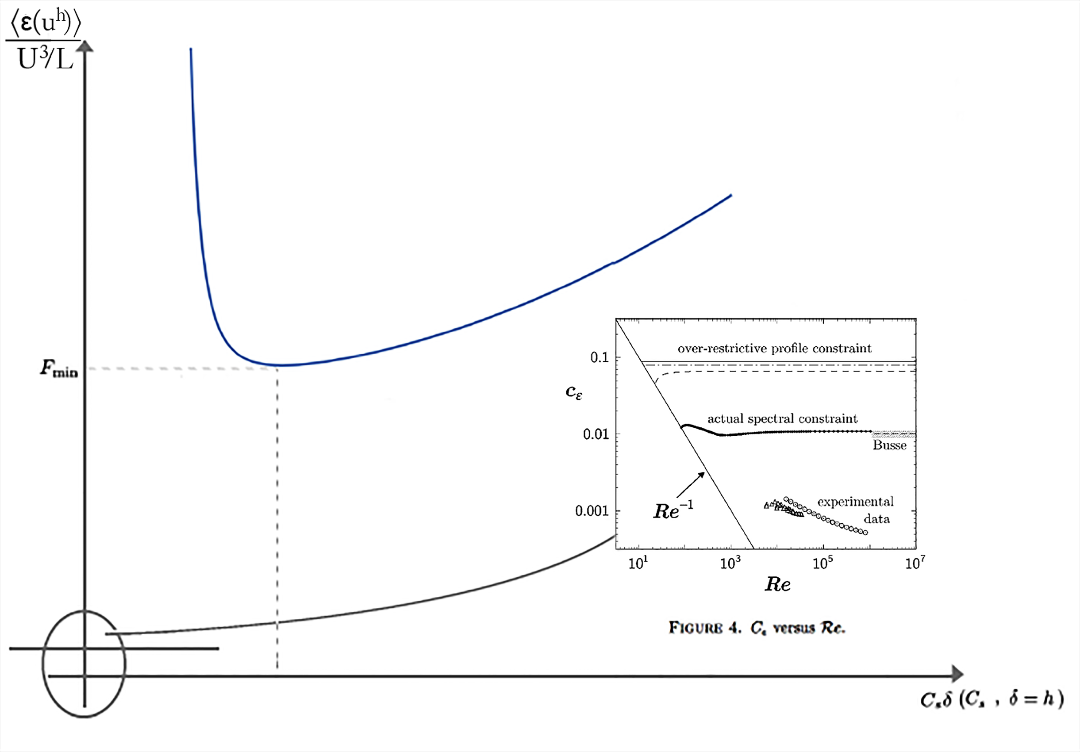}
\caption{Theorem \ref{thm2} vs experimental results in Figure \ref{Cepsilon}.}
\label{Campare}
\end{figure}

\begin{cor}\label{Cor2}
Let the mesh size $h \geq \mathcal{O} (\Rey^{-1}) \, L$ be fixed. Then for any choice of $C_s > 0$ and $\delta>0$, the upper estimate on $ \frac{\langle  \varepsilon (u^h)\rangle }{U^3/L} $ in the Theorem \ref{thm2} is larger than the dissipation coefficient $C_{\epsilon}$ in Figure \ref{Cepsilon}.  
\end{cor}
\begin{proof}
Solving the minimization problem $(C_s \delta)_{\min}= \operatorname*{arg\,min}_{C_s \delta} G(C_s \delta)$, the minimum of the function, 
\begin{equation}\label{101}
 G(C_s \delta)=  \frac{1}{\Rey} \,\frac{L}{h}+   (\frac{C_s\, \delta}{h})^2+ \frac{L^5}{(C_s \delta)^4\,h}+\frac{L^{\frac{5}{2}}}{(C_s\, \delta)^{4}}\, {h^{\frac{3}{2}}},
 \end{equation}
occurs at, 

$$F_{\min}= \frac{1}{\Rey} \,\frac{L}{h}+ 2 \big[(\frac{L}{h})^5+ (\frac{L}{h})^{\frac{5}{2}}\big]^{\frac{1}{3}},$$ 
for $(C_s\delta)_{\min} = \sqrt[6]{2}\,h\, [(\frac{L}{h})^5+(\frac{L}{h})^{\frac{5}{2}}]^\frac{1}{6}$ assuming $h$  and $L$ are fixed. This minimum amount  which is much larger than the experimental range of  $C_{\epsilon}$, $\frac{1}{\Rey}\leq C_{\epsilon} \leq 0.1 $, leads to the  over-dissipation of the model, Figure \ref{Campare}.
\end{proof}

\begin{re}
The  suggested extra dissipation in the Corollary \ref{Cor1} and \ref{Cor2}, which is consistent with experience with the Smagorinsky model  (e.g., Iliescu
and Fischer \cite{IF04} and Moin and Kim \cite{MK82}), can laminarize the numerical approximation of a turbulent flow and prevent the transition to turbulence.
\end{re}

\begin{re}

 In the limit of high Reynolds number  and for a fixed computational cost $h \geq \mathcal{O}(\Rey^{-1}) \, L$, the estimate in  Theorem \ref{thm2} is the function of $C_s \delta$, 
 
  $$ \frac{\langle  \varepsilon (u^h)\rangle }{U^3/L} \simeq  (\frac{C_s\, \delta}{h})^2+ \frac{1}{(C_s \delta)^4}\, [\frac{L^5}{h}+L^{\frac{5}{2}}\, {h^{\frac{3}{2}}}],$$
which consists of two major parts. First,
$$\lambda_2= \mathcal{O}((C_s\delta)^2)$$
  is derived from the eddy viscosity term and is quadratic in $(C_s \delta)$. The latter, 
  $$\lambda_3= \mathcal{O}((C_s\delta)^{-4})$$ 
  is derived   from the non-linearity term and is  inversely proportional  to $(C_s \delta)^{4}$. Therefore, the behavior of the graph $\frac{\langle  \varepsilon (u^h)\rangle }{U^3/L}$  in  Figure \ref{Campare} is decided by the competition between the increasing function $\lambda_2$ and the decreasing function $\lambda_3$. This observation   suggests model over-dissipation is due to the action of the model viscosity, which is also consistent with \cite{L16}. Analysis in \cite{L16} suggested that the model over-dissipation is due to the action of the model viscosity in boundary layers rather than in interior small scales generated by the turbulent cascade.
 \end{re}

 \begin{figure}[b]
  \centering
  \begin{minipage}[b]{0.4\textwidth}
 \includegraphics[scale=0.26]{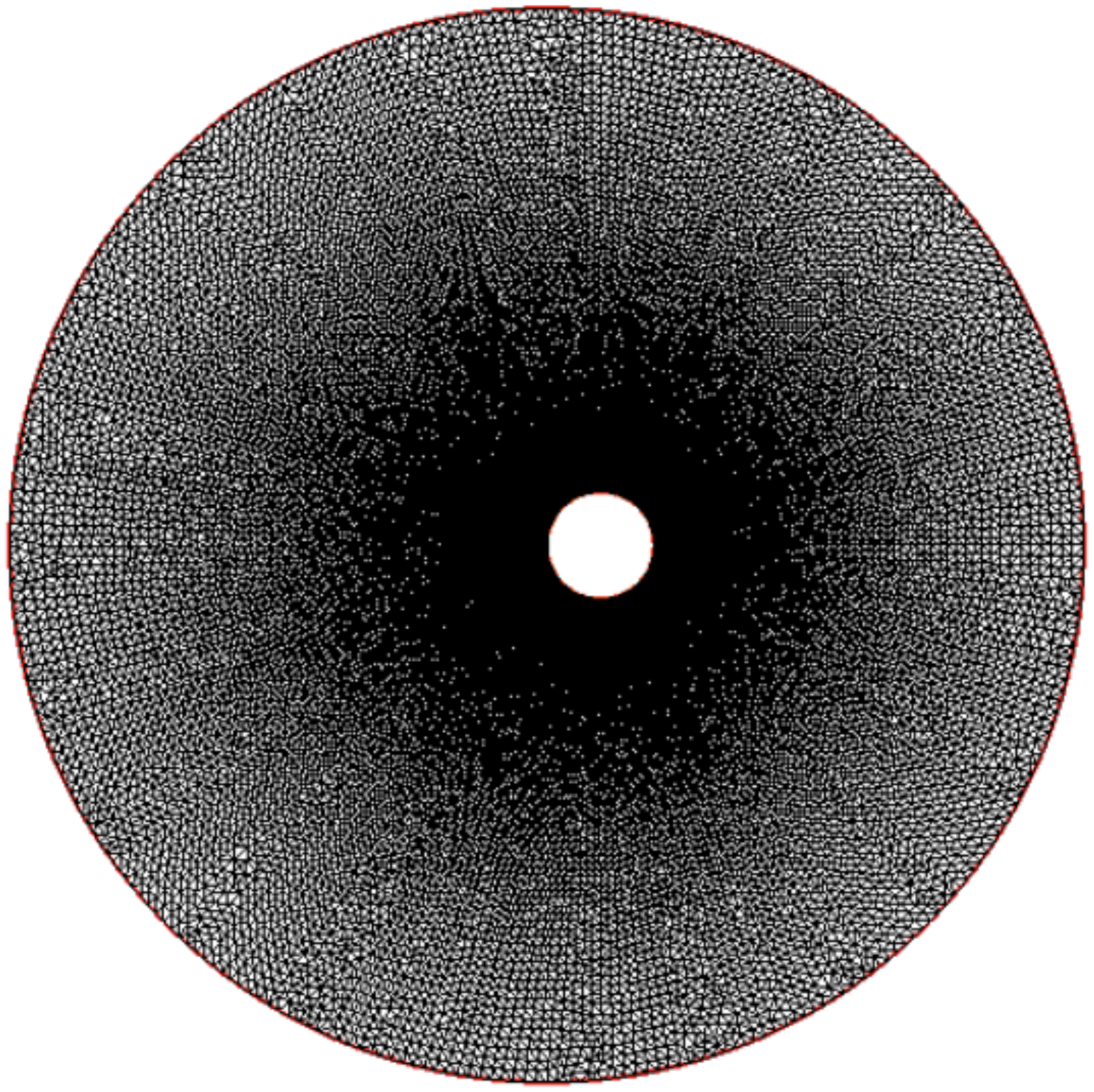}
\caption{m=300, n=200.}
\label{FineMesh}
  \end{minipage}
  \hfill
  \begin{minipage}[b]{0.4\textwidth}
    \includegraphics[scale=0.27]{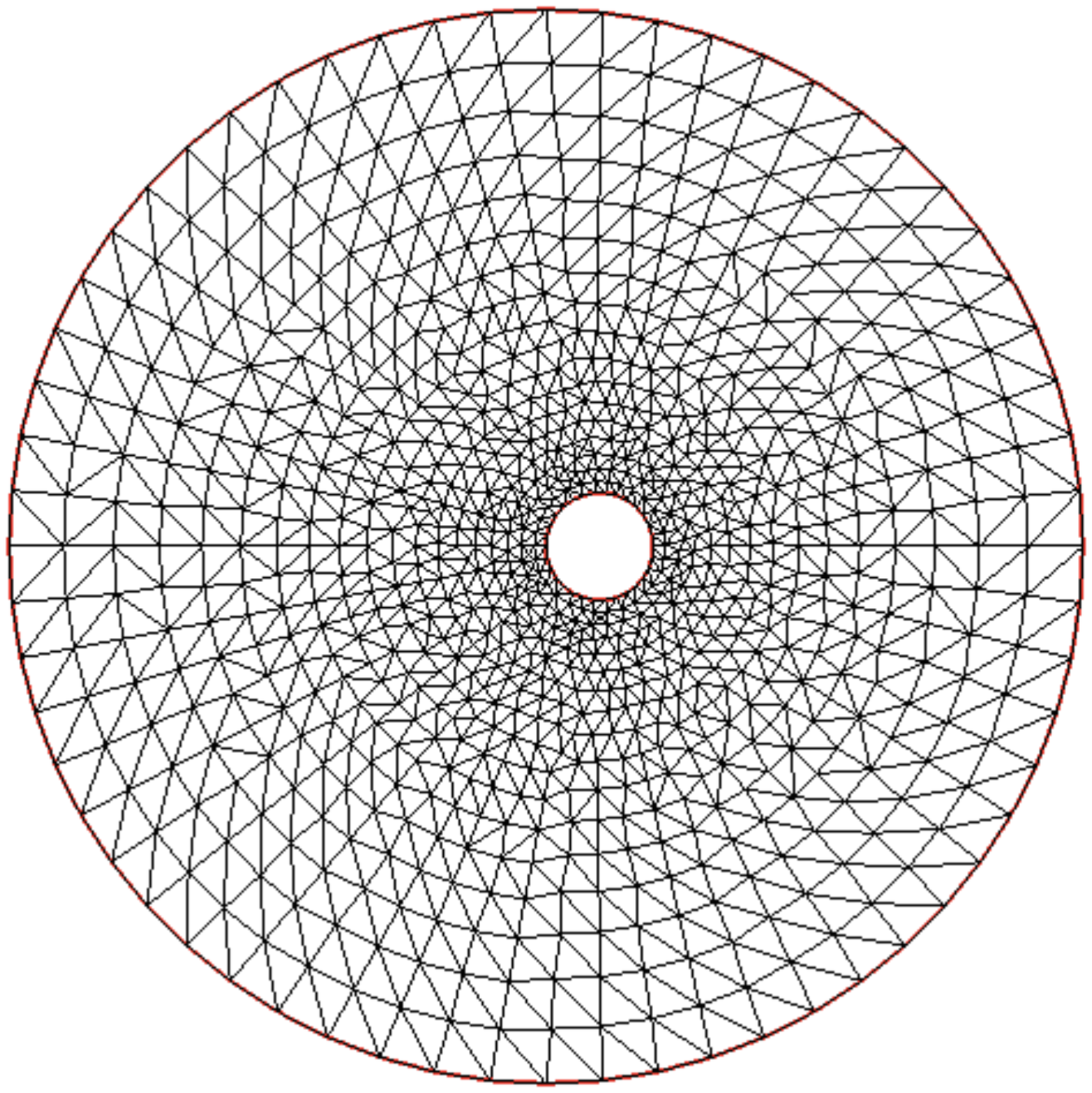}
    \caption{m=60, n=30}
    \label{CoarseMesh}
  \end{minipage}
  
\end{figure}

\section{Numerical Illustration}

The test is a  comparison between simulation of the NSE and the Smagorinsky model for two dimensional  time-dependent shear flow  between two  cylinders, motivated by the classical problem of flow between rotating cylinders. Motivated from the Taylor experiment, at high enough $\Rey$ we expect a steady flow lose its stability and the flow becomes  time-periodic. At still higher $\Rey$ a very complex and fully turbulence flow should be observed  (p. 91 of \cite{L08}). 

\begin{figure}[b]
  \centering
   \begin{minipage}{0.32\textwidth}
    \includegraphics[scale=0.2]{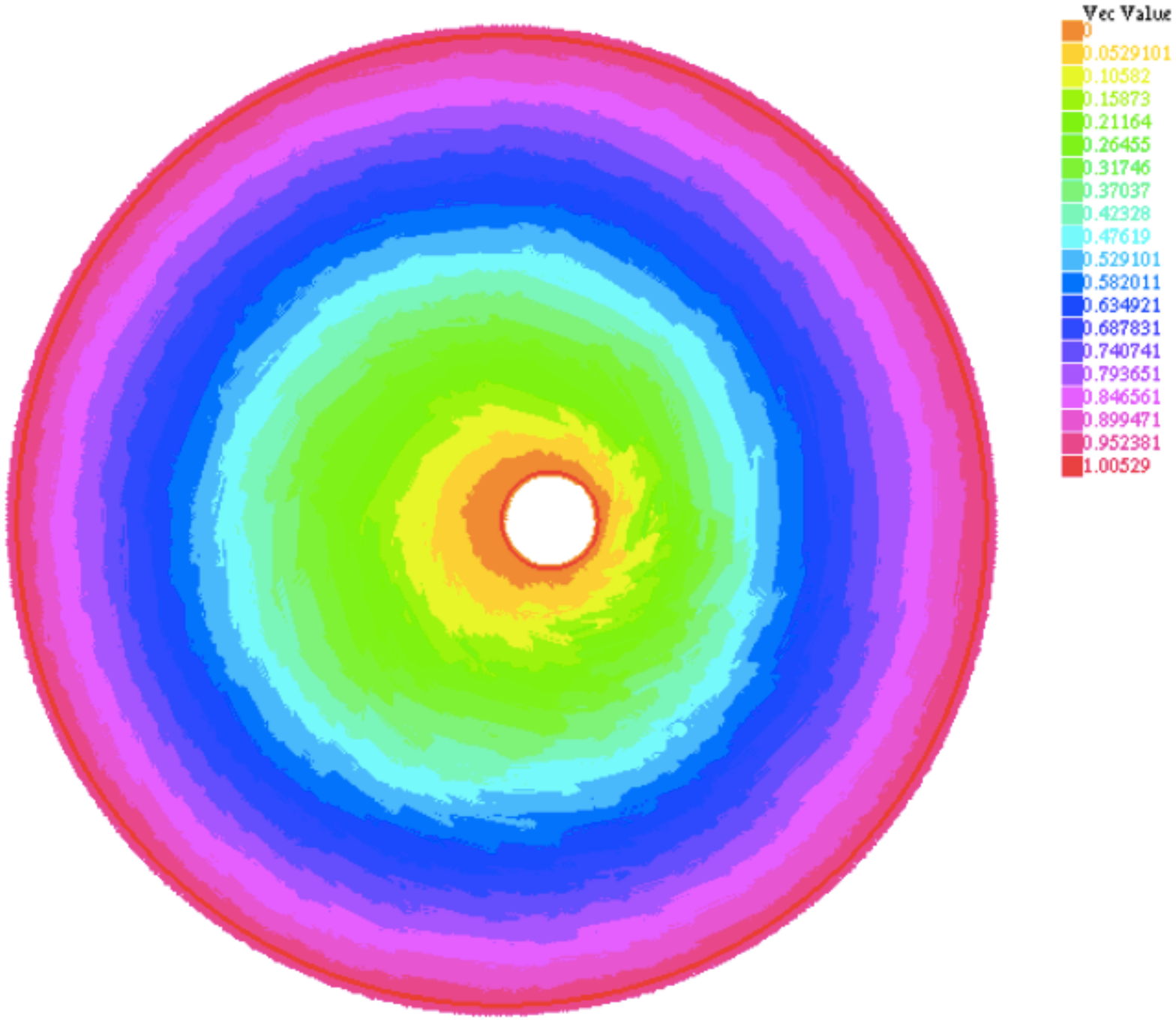}
    \caption{NSE; $T=0$.}
    \label{NSET0}
      \end{minipage}
 \vfill
  \begin{minipage}{0.32\textwidth}
    \includegraphics[scale=0.2]{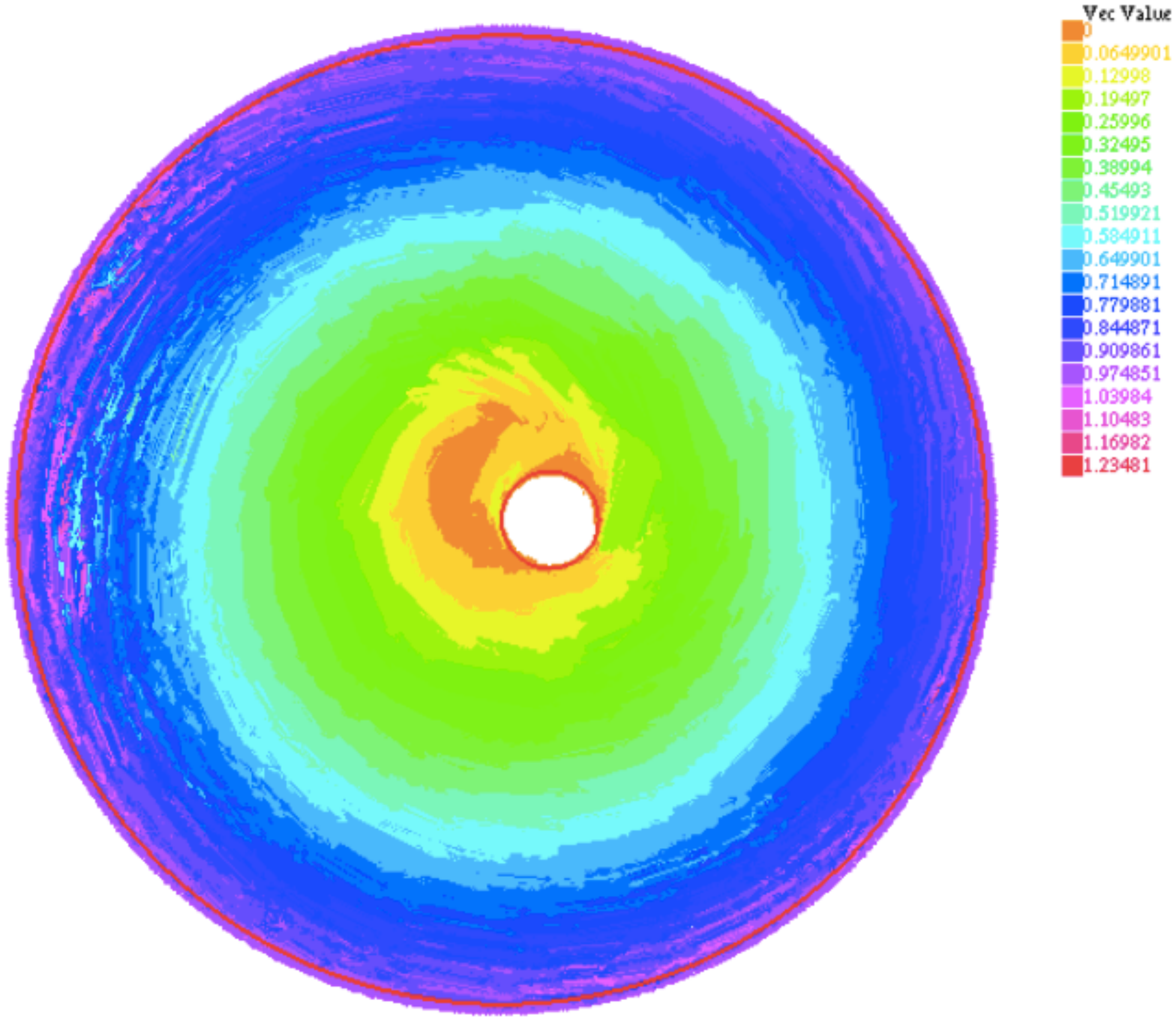}
    \caption{NSE; $T=1$.}
    \label{NSET1}
      \end{minipage}
 \hfill
 \begin{minipage}{0.32\textwidth}
    \includegraphics[scale=0.2]{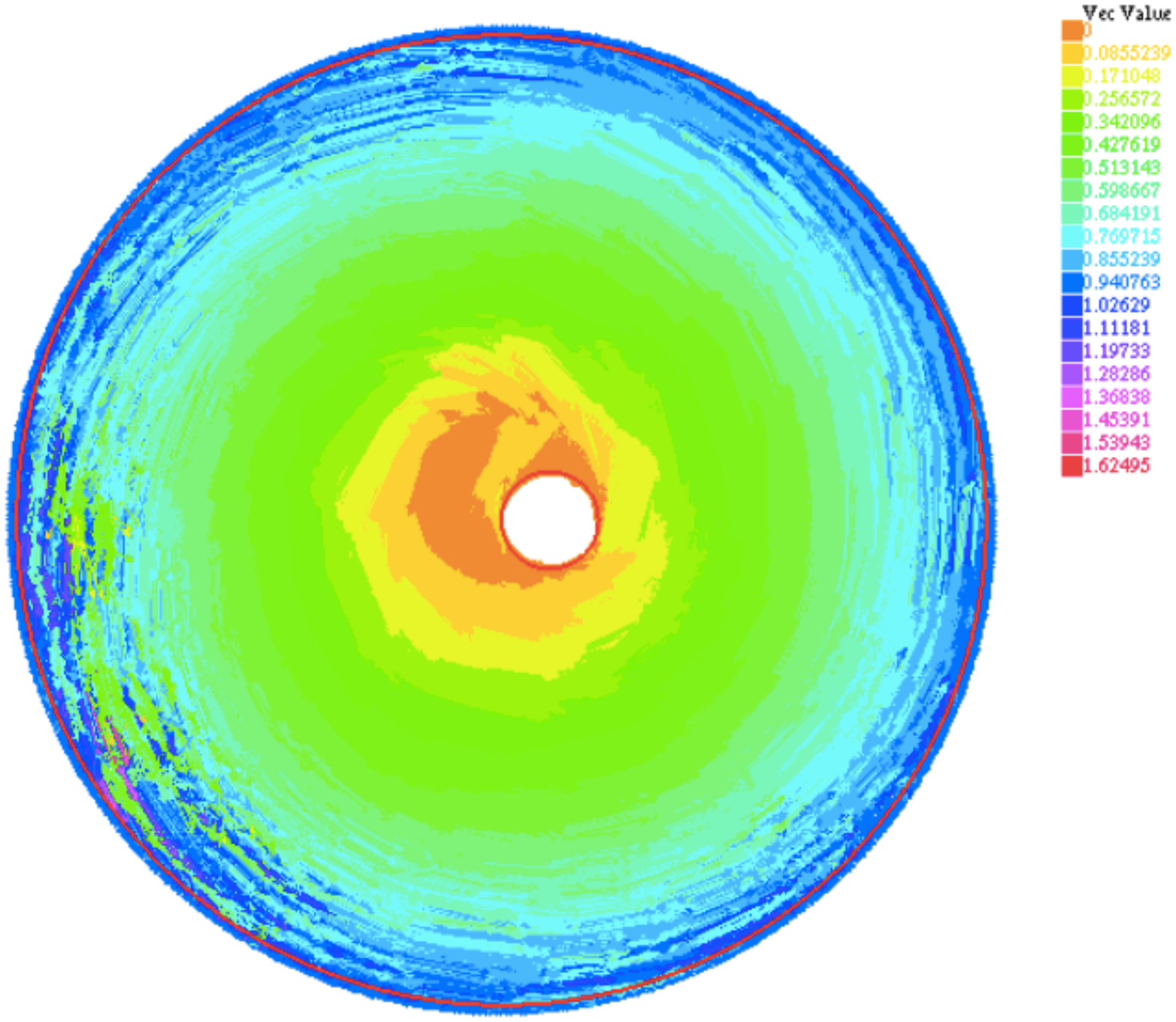}
    \caption{NSE; $T=3$.}
    \label{NSET3}
  \end{minipage}
   \hfill
 \begin{minipage}{0.32\textwidth}
    \includegraphics[scale=0.2]{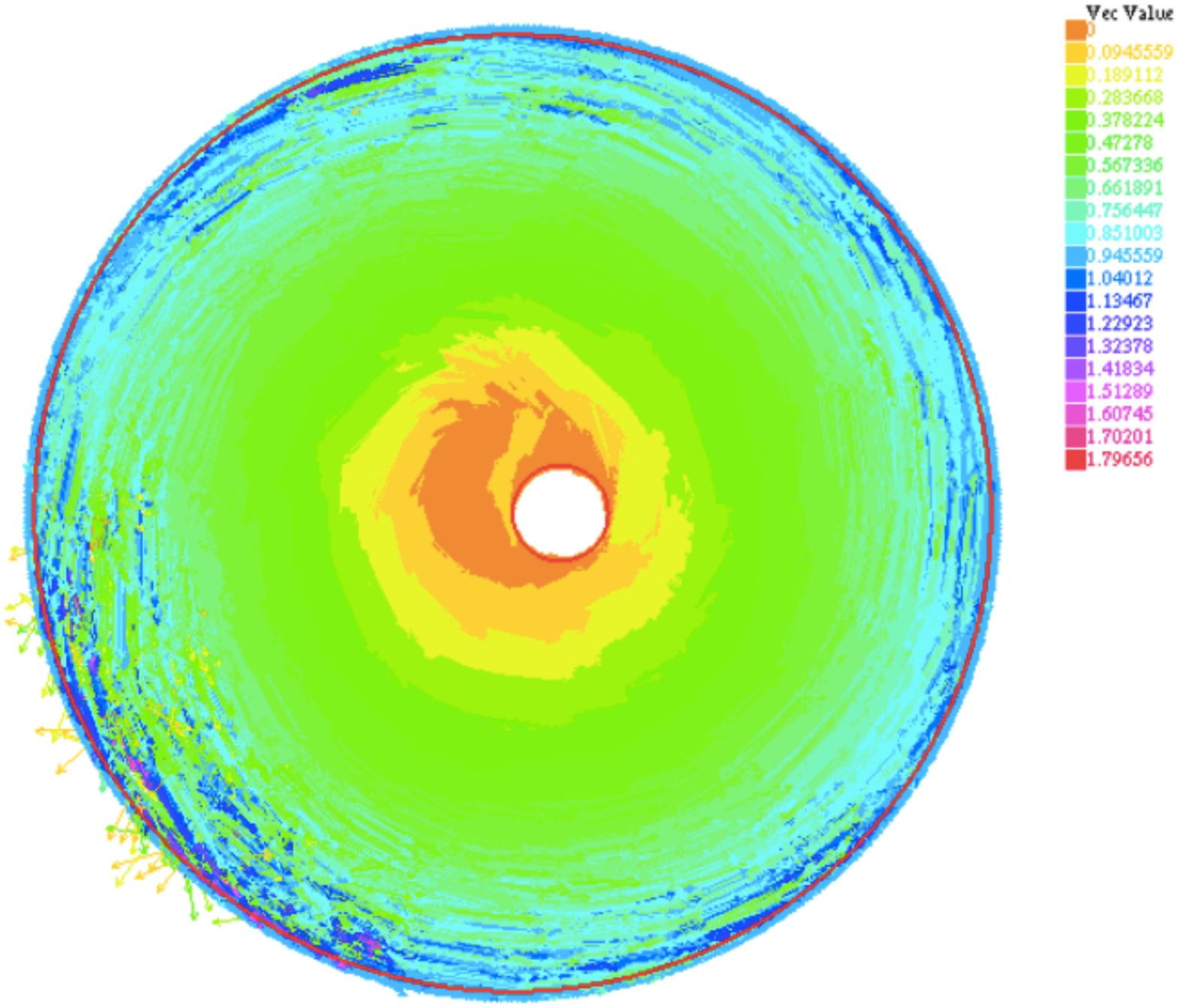}
    \caption{NSE; $T=5$.}
    \label{NSET5}
   \end{minipage}
   \vfill
    \begin{minipage}{0.32\textwidth}
    \includegraphics[scale=0.2]{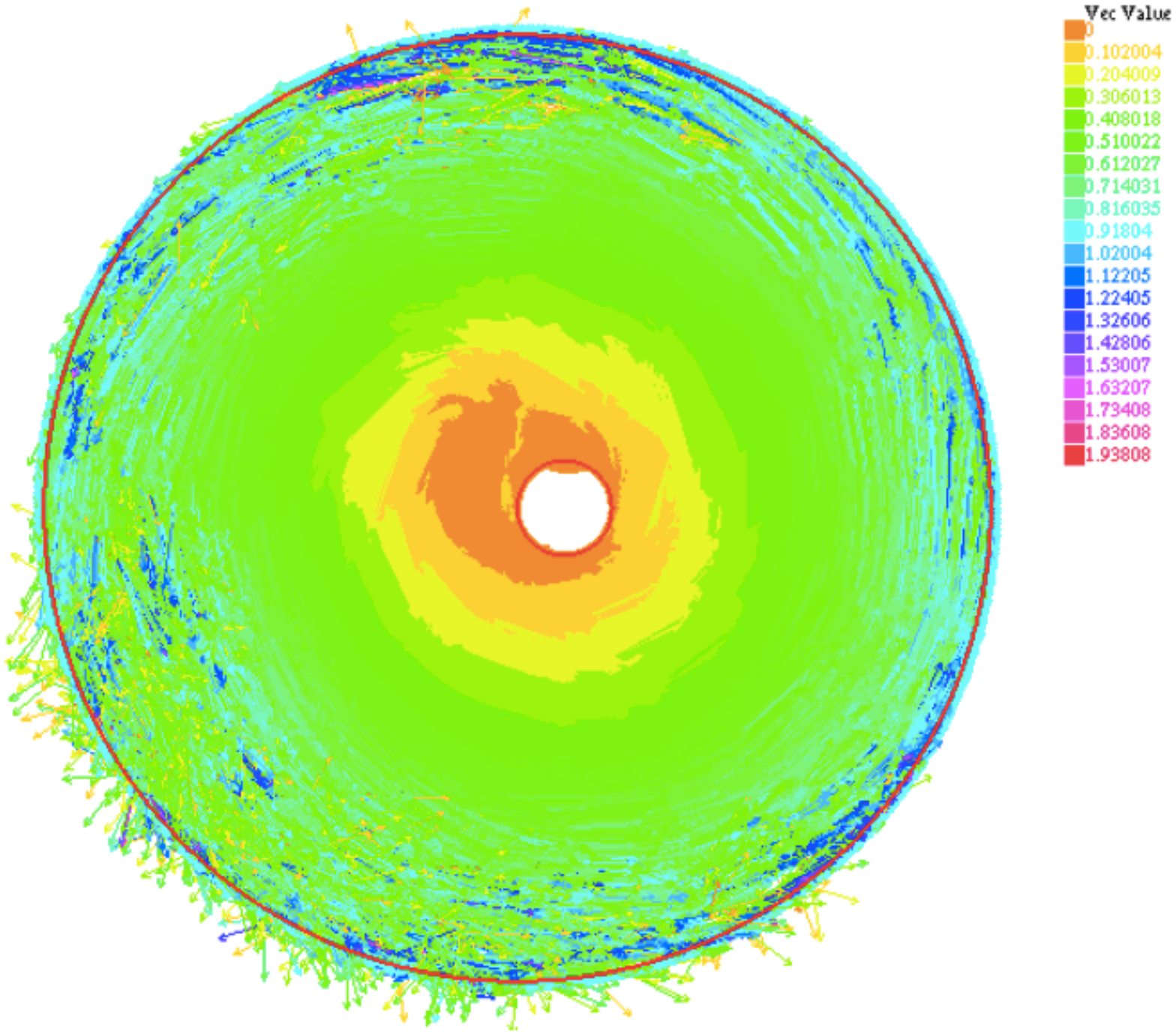}
    \caption{NSE; $T=7$.}
    \label{NSET7}
      \end{minipage}
 \hfill
 \begin{minipage}{0.32\textwidth}
    \includegraphics[scale=0.2]{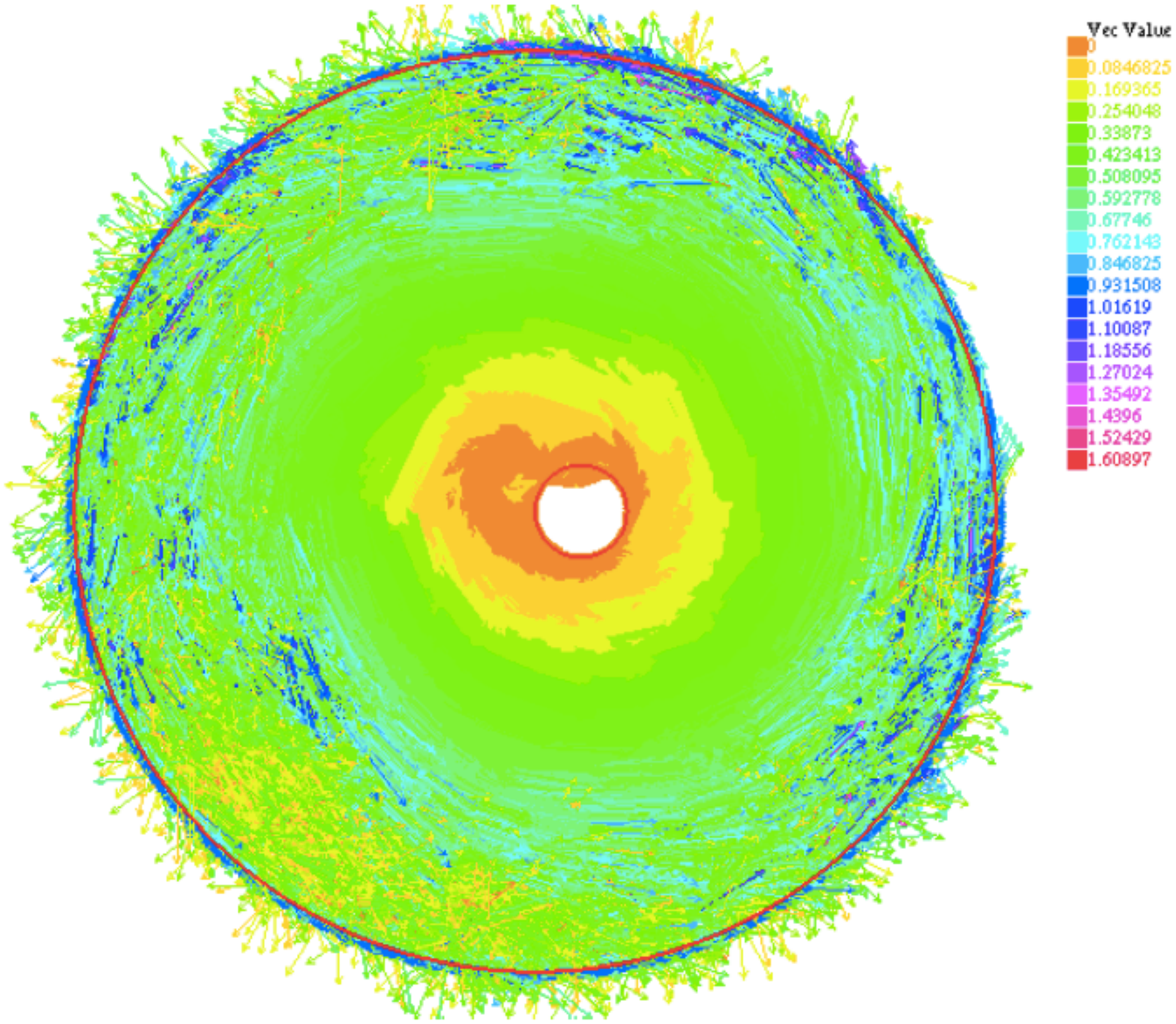}
    \caption{NSE; $T=9$.}
    \label{NSET9}
  \end{minipage}
   \hfill
 \begin{minipage}{0.32\textwidth}
    \includegraphics[scale=0.2]{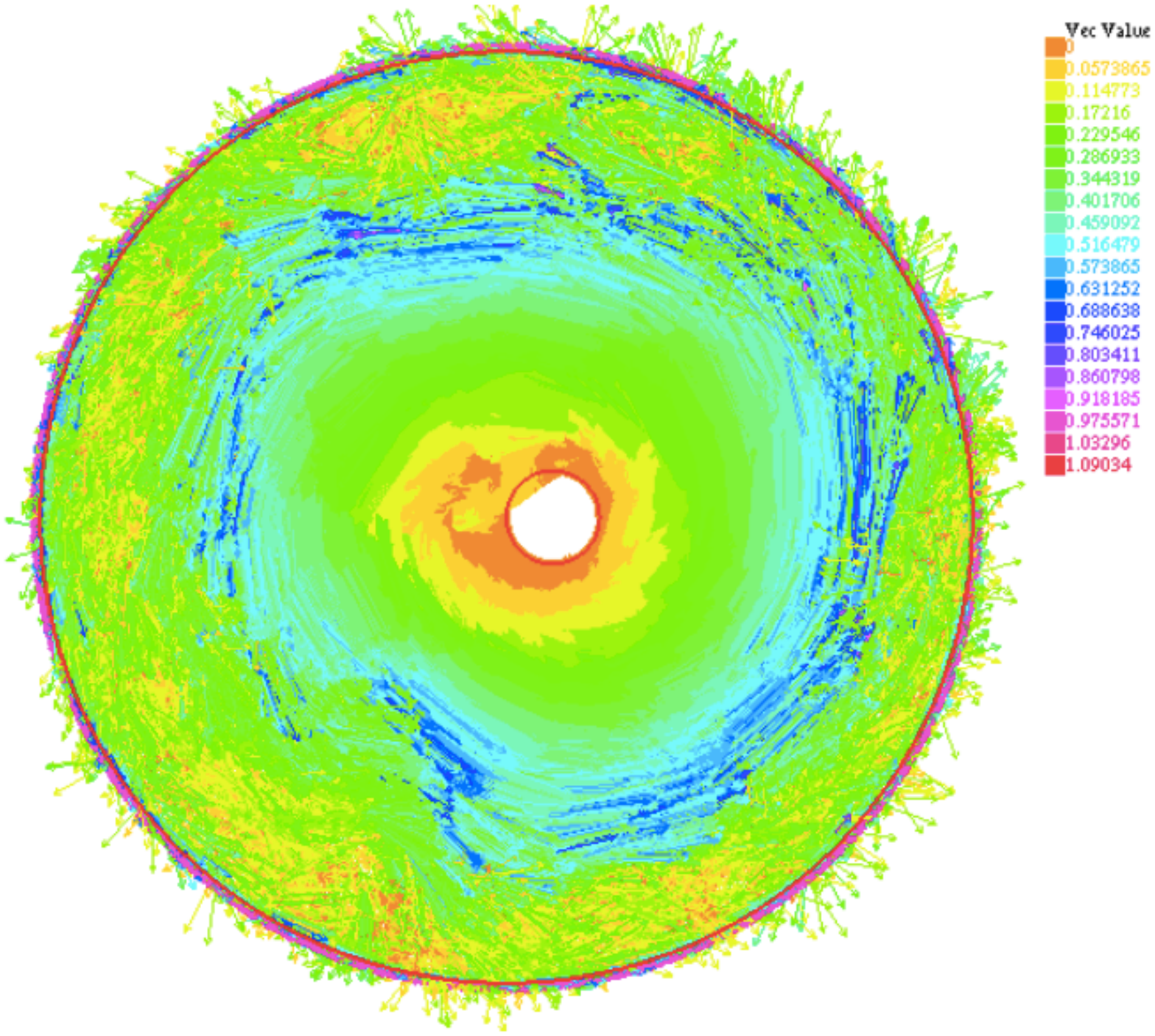}
    \caption{NSE; $T=10$.}
    \label{NSET10}
   \end{minipage}
    \end{figure}

The domain is a disk with a smaller off-centered  obstacle inside. The flow is driven by the rotational force at the outer circle in an absence of body force, with no-slip boundary conditions on the inner circle. The tests were performed using FreeFEM++  \cite{H12}, with Taylor-Hood elements (continuous piecewise quadratic polynomials for the velocity and continuous linear polynomials for the pressure) in all tests. We take $\Rey=4500$, final time $T=10$  and time step $\Delta t=0.01$.   The initial condition $u_0$ is generated by solving the steady Stokes problem with the same condition described above on the same geometry, this gives an  initial condition that is divergence free and satisfies the boundary  conditions (Figure \ref{NSET0}). The  Backward Euler method is utilized for  time discretization.

 For the  resolved NSE simulation, which is our \textit{truth} solution,   the outer and inner circles are discretized with 300 and 200 points respectively. The mesh is extended to all of the domain  as a Delaunay mesh (Figure \ref{FineMesh}). As  expected, flow does not approach a steady state and  the unsteady behavior of the velocity field especially near the boundaries can be seen (Figures \ref{NSET1} : \ref{NSET10}).  Note that in the absence of body force here, the most distinctive feature of this flow is due to the interaction of the flow with the outer boundary.

However, if the simulation is performed on a coarser mesh, it makes sense to try the Smagorinsky model  since the turbulence model is introduced to account for sub-mesh scale effects.  The under-resolved  mesh is parameterized by the number of mesh points m = 60 around the outer circle and n = 30 mesh points  around the inner circle and extended to all of the domain  as a Delaunay mesh, (Figure \ref{CoarseMesh}). Providing all other conditions hold as the NSE test,  we used $C_s= 0.17$  in the computations presented in this section (see \cite{L67}).   Motivated by Corollaries \ref{Cor1} and \ref{Cor2},  we consider  $\delta = h$  and $\delta= h^{\frac{1}{6}}$.  In both cases,  the model predicts the flow will quickly reach a nonphysical equilibrium, which  is clearly over-dissipated. This extra dissipation  laminarizes the  approximation of the flow  and prevents the transition to turbulence (Figures  \ref{SMT1}:\ref{SMT10}).


\begin{figure}[b]
  \centering
  \begin{minipage}{0.32\textwidth}
    \includegraphics[scale=0.23]{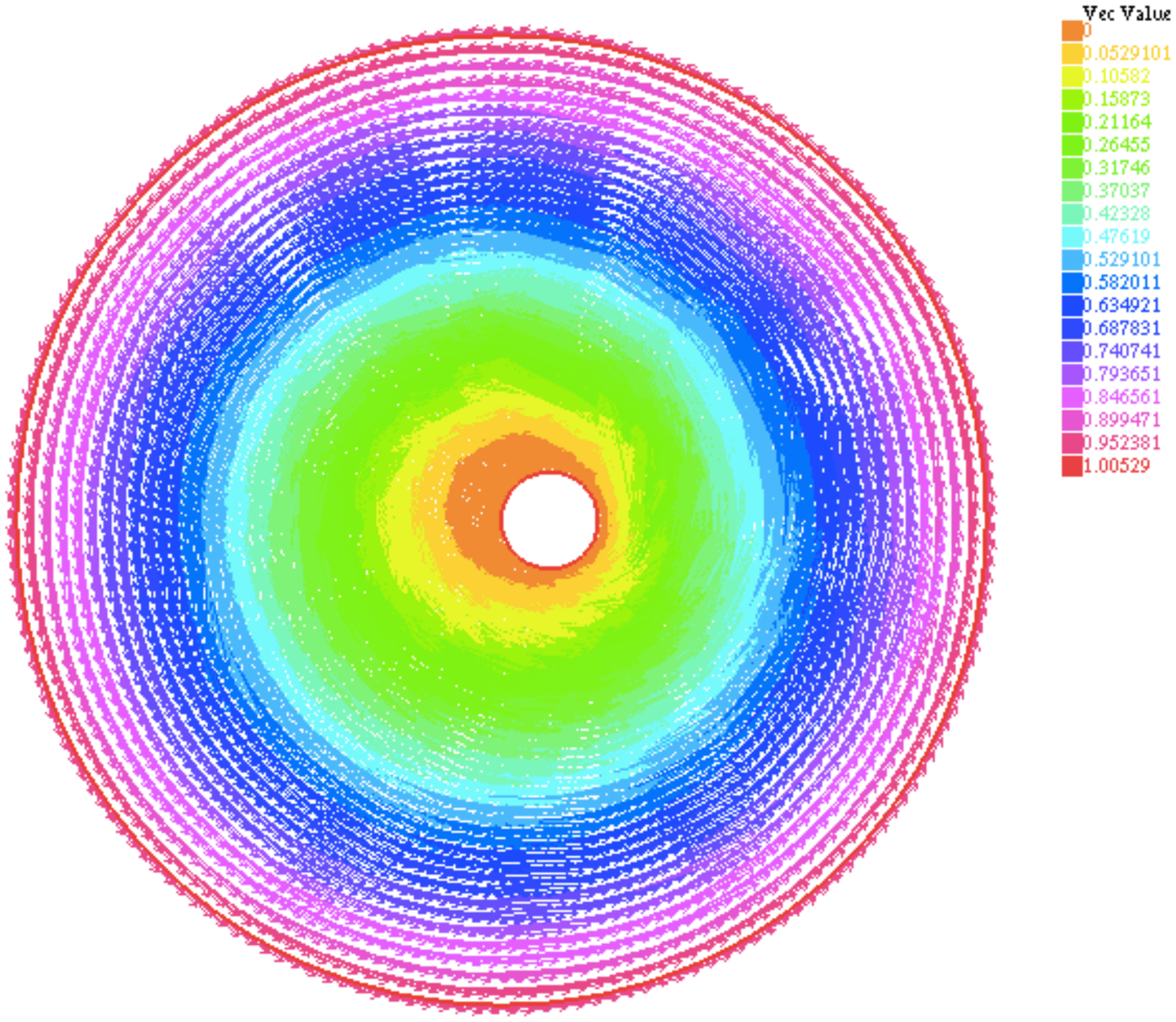}
    \caption{SM; $T=1$.}
    \label{SMT1}
      \end{minipage}
 \hfill
 \begin{minipage}{0.32\textwidth}
    \includegraphics[scale=0.23]{SM.pdf}
    \caption{SM; $T=7$.}
    \label{SMT7}
  \end{minipage}
   \hfill
 \begin{minipage}{0.32\textwidth}
    \includegraphics[scale=0.23]{SM.pdf}
    \caption{SM; $T=10$.}
    \label{SMT10}
   \end{minipage}
   
 \end{figure}

\section{Conclusion}

We investigate the \textit{computed}  time-averaged energy dissipation $\langle  \varepsilon (u^h)\rangle$  here for a shear flow turbulence on an under-resolved mesh. $\langle  \varepsilon (u^h)\rangle$ scales as the equilibrium dissipation law, $\frac{U^3}{L}$, for the under-resolved mesh independent of $\nu$ at high Reynolds number being considered (Theorem \ref{thm2}). The upper bound in Theorem \ref{thm2} does not give the correct dissipation for any choice of the Smagorinsky constant  $C_s>0$ and  filter size $\delta>0$ (Corollaries  \ref{Cor1} and \ref{Cor2}).  Comparing the results pictured in Figures  \ref{SMT1}:\ref{SMT10} with the NSE simulations (Figures \ref{NSET1} : \ref{NSET10}) give strong, although admittedly very preliminary, support for the fact that  the model introduces too much diffusion into the flow for any artificial parameters $C_s >0$ and $\delta$, which is consistent with Corollaries \ref{Cor1} and \ref{Cor2}.  

In addition, it is shown that the viscosity term $\nu \Delta u$ does not affect $\langle  \varepsilon (u^h)\rangle$ for any choice of discretization parameters. The next important step would be analyzing the computed energy dissipation rate for other turbulent models, e.g. \cite{NNGLP17}, on an under-resolved spatial  mesh.

The analysis here indicates that the model over-dissipation is due to the action of the eddy viscosity term  $\grad \x ((C_s \delta)\,|\grad u|\, \grad u)$.  The classical approach to correct the over-dissipation of the Smagorinsky model is to multiply the eddy viscosity term  with a damping function \cite{Van}. The mathematical analysis in \cite{A.P}  shows that the combination of the Smagorinsky with the van Driest damping does not over dissipate assuming infinite resolution (i.e. continues case). The unexplored question is: What is the statistics of this combination on an under-resolved mesh size?

Wang \cite{W00} shows the smallness of the energy dissipation rate due to the tangential derivative directly from the NSE for shear driven flow. Motivated by his work, one can study the effect of tangential derivative in the boundary layer for the combination of the Smagorinsky model and the damping function. Moreover, taking advantage of the analysis in  \cite{W00}  the energy  dissipation rate for the discretized NSE on the coarse mesh can be investigated. The open question  is how to  prove the boundedness of the kinetic energy of the approximate velocity, $||u^h||^2$, for the  NSE without any restriction on the mesh size $h$.

\bigskip

\noindent{\bf Acknowledgments.} 
I am extremely grateful to Professor William Layton  from whom I have  learned so much about fluid dynamics and numerical analysis. I would also like to acknowledge an unknown referee whose suggestions greatly helped to improve this paper. A.P. was Partially supported by NSF grants, DMS 1522267 and CBET 1609120.

\end{document}